\documentclass[11pt,reqno]{amsart}
\usepackage{amsfonts,latexsym,amsmath, amssymb}
\usepackage{mathrsfs}
\usepackage{bm}
\usepackage{xcolor}
\usepackage{color}
\usepackage{amsmath,nccmath}
\usepackage{cite}
\usepackage[utf8]{inputenc}
\usepackage{graphicx,enumerate}
\usepackage{ulem}
\usepackage{enumitem}
\usepackage{epsfig}
\usepackage{graphicx}
 \usepackage{graphics}

\newcommand{\bea}{\begin{eqnarray}}
\newcommand{\eea}{\end{eqnarray}}
\def\beaa{\begin{eqnarray*}}
\def\eeaa{\end{eqnarray*}}
\def\ba{\begin{array}}
\def\ea{\end{array}}
\def\be#1{\begin{equation} \label{#1}}
\def \eeq{\end{equation}}

\newcommand{\gap}{\hspace{1pt}}

\def\be{{\beta}}

\def\R{{\mathbb{R}}}

\def\N{{\mathbb N}}

\def\sgH2{\sigma_H^2}
\def\sgL2{\sigma_L^2}

\newtheorem{theorem}{Theorem}[section]
\newtheorem{lemma}[theorem]{Lemma}
\newtheorem{proposition}[theorem]{Proposition}
\newtheorem{corollary}[theorem]{Corollary}
\newtheorem{definition}[theorem]{Definition}
\newtheorem{remark}[theorem]{Remark}

\newtheorem{example}[theorem]{Example}

\numberwithin{equation}{section}

\numberwithin{equation}{section}

\allowdisplaybreaks
\begin{document}

\title[Free interface problem with fractional order kinetics]{Existence of a traveling wave solution in a free interface problem with fractional order kinetics}
\author[C.-M. Brauner. R. Roussarie, P. Shang, L. Zhang]{Claude-Michel Brauner, Robert Roussarie, Peipei Shang, Linwan Zhang}
\address{School of Mathematical Sciences, Tongji University, 1239 Siping Rd., Shanghai 200092, China, and Institut de Math\'ematiques de Bordeaux UMR CNRS 5251, universit\'e de Bordeaux, 33405 Talence Cedex, France.}
\address{Institut de Math\'ematique de Bourgogne UMR CNRS 5584, universit\'e de Bourgogne-Franche Comt\'e, B.P. 47870
	21078 Dijon Cedex, France}
\address{School of Mathematical Sciences, Tongji University, 1239 Siping Rd., Shanghai 200092, China.}
\address{School of Mathematical Sciences, Tongji University, 1239 Siping Rd., Shanghai 200092, China.}
\email{claude-michel.brauner@u-bordeaux.fr}
\email{Robert.Roussarie@u-bourgogne.fr}
\email{shang@tongji.edu.cn}
\email{1810879@tongji.edu.cn}
\thanks{$^*$ Corresponding author: Peipei SHANG (shang@tongji.edu.cn)}


\begin{abstract} In this paper we consider a system of two reaction-diffusion equations that models
diffusional-thermal combustion with stepwise ignition-temperature kinetics and fractional reaction
order $0 < \alpha< 1$. We turn the free interface problem into a scalar free boundary problem coupled with an integral equation. The main intermediary step is to reduce the scalar problem to the study of a non-$C^1$ vector field in dimension $2$. The latter is treated by qualitative topological methods based on the Poincaré-Bendixson Theorem. The phase portrait is determined and the existence of a stable manifold at the origin is proved. A significant result is that the settling time to reach the origin is finite, meaning that the trailing interface is finite in contrast to the case $\alpha=1$, but in accordance with $\alpha=0$. Finally, the integro-differential system is solved via a fixed-point method.
	\end{abstract}
\maketitle
\section{Introduction and Main Results}
A flame spreading through a motionless gas may be described by a system of two parabolic nonlinear equations for the normalized temperature, $T$, and the concentration of deficient reactant, $Y$.  This system, in non-dimensional form, reads:
\begin{eqnarray}\label{evolution_system}
\left\{
\begin{array}{l}
T_t=\Delta T+W(T,Y), \\[2mm]
\displaystyle Y_t=\Lambda  \Delta Y-W(T,Y),
\end{array}
\right.
\end{eqnarray}
where $\Lambda $ (the ratio of mass diffusivity and thermal diffusivity) is the inverse of the Lewis number  and  $W(T,Y)$ is the reaction rate.

In this paper, the reaction rate is assumed to be of stepwise ignition type, that is, $W(T,Y)\equiv 0$  when the normalized temperature $T$ is below the ignition temperature $\theta \in (0,1)$,
and $W(T,Y)$ depends only on the concentration $Y$ when the temperature is above $\theta$. The recent renewed interest in stepwise ignition temperature kinetics is due to its applicability to the studies of overall effective chemical kinetics of certain reactive mixtures
(see Brailovsky, Gordon, Kagan and Sivashinsky \cite{BGKS15} and references therein).

In \cite[Section 2]{BGKS15}, two classes of stepwise kinetics are considered.
\begin{enumerate}[label=(\arabic*), wide, labelwidth=!, labelindent=0pt]
	\item [\rm (i)] Zero-order stepwise temperature kinetics with reaction rate given by
	\begin{eqnarray}\label{zero_order}
	W_0(T,Y)=
	\left\{
	\begin{array}{lllll}
	1,  & \mbox{if} &T\geqslant \theta & \mbox{and}& Y>0, \\[2mm]
	0, & \mbox{if} &T<\theta & \mbox{and/or} & Y=0.
	\end{array}
	\right.
	\end{eqnarray}
	\item [\rm (ii)] First-order stepwise temperature kinetics with reaction rate given by
	\begin{eqnarray}\label{first_order}
	W_1(T,Y)=
	\left\{
	\begin{array}{lllll}
	Y,  & \mbox{if} & T\geqslant \theta, \\[2mm]
	0, & \mbox{if} &T<\theta,
	\end{array}
	\right.
	\end{eqnarray}
	or, equivalently, $W_1(T,Y)= Y H(T- \theta)$ where $H$ is the Heaviside function.
\end{enumerate}

It was shown in \cite{BGKS15} that problem \eqref{evolution_system} is equivalent to a free
interface problem in both cases. However, the resulting free interface problems for these two
cases are very different. To clarify this difference, it is convenient to consider problem \eqref{evolution_system} in a two-dimensional strip $\mathbb{R} \times (-{\ell}/{2},{\ell}/{2})$; the spatial coordinates are denoted by $(x,y)$, $t>0$ is the time.

First, in the case of the zero-order kinetics \eqref{zero_order}, there are two free interfaces:
the \textsl{ignition interface} $x=f(t,y)$, also called the flame front, defined by
\begin{equation}
\label{ignition}
T(t, f(t,y),y) = \theta,
\end{equation}
and the \textsl{trailing interface}  $x = g(t,y)$, $g(t,y)>f(t,y)$, defined by
\begin{equation}
\label{trailing}
Y(t, g(t,y),y) = 0.
\end{equation}
Then, the system for $(T,Y)$ reads as follows
\begin{eqnarray}\label{FBP_zero}
\left\{
\begin{array}{ll}
T_t= \Delta T, \;
Y_t= \Lambda\Delta Y, \quad &x<f(t,y),\\[2mm]
T_t = \Delta T + 1, \;
Y_t = \Lambda \Delta Y - 1, \quad &f(t,y)<x<g(t,y), \\[2mm]
T_t = \Delta T, \;
Y(t,x,y)=0, \quad &x>g(t,y).
\end{array}
\right.
\end{eqnarray}
The functions $T$ and $Y$ and their normal derivatives are continuous across the ignition and trailing free interfaces. As $x \to \pm\infty$, the following holds:
\begin{equation}\label{boundary_conditions_0}
T(t, -\infty,y)=0, \quad Y(t,-\infty,y)= T(t,+\infty,y)=1.
\end{equation}

Second, in the case of first-order stepwise kinetics, the system reads:
\begin{eqnarray}
\label{FBP_first}
&&\left\{\begin{aligned}
&T_t=\Delta T, \;
Y_t=\Lambda \Delta Y,\quad &x<f(t,y),\\[2mm]
	&T_t=\Delta T+Y, \;
Y_t=\Lambda \Delta Y-Y,\quad &x>f(t,y).
\end{aligned}\right.
\end{eqnarray}
Therefore, only the ignition interface is involved, at which $T$ and $Y$ and their normal derivatives are continuous. As $x \to \pm\infty$,
\begin{equation}\label{boundary_conditions_1}
T(t,-\infty,y)=Y(t,+\infty,y)=0, \quad Y(t,-\infty,y)= T(t,+\infty,y)=1.
\end{equation}

Such drastic difference in qualitative properties of traveling front solution for zero- and first-
order reaction models raises a natural question: why two free interfaces are generated in the case of zero-order stepwise kinetics \eqref{zero_order} and just one in first-order kinetics \eqref{first_order}.
We aim to understand the role of the reaction rate in the process; to this end, we consider the gamut of intermediate combustion systems when the order of the reaction
(which we will denote by $\alpha$) is between $0$ and $1$.
It is important to note that
this question is interesting not only from mathematical viewpoint but also from perspective of
applications in combustion. Indeed, as evident from experimental observation, the overall
reaction order can change quite substantially with the equivalence ratio (i.e. the ratio of fuel
and oxidizer in the mixture) as well as pressure. For example, in the not too extreme case of
hydrogen-air flame with equivalence ratio $3$, the overall reaction rate of $\alpha = 1$ at pressure $20$
atm drops to $\alpha = 0.3$ at $50$ atm (see \cite[p.~280]{Law10}).

The main goal of this paper is to understand the role of the reaction order $0 < \alpha< 1$
on qualitative properties of planar traveling fronts for free interface problems as well as their
connection with the limiting systems associated with the cases $\alpha = 0$ and $\alpha = 1$.
Mathematically, we fill the gap between the zero-order and first-order kinetics. Mimicking \eqref{first_order}, we define an ``$\alpha$-order reaction rate'' for $0<\alpha<1$:
\begin{eqnarray}\label{alpha-order kinetics}
W_{\alpha}(T,Y)=Y^{\alpha} H(T-\theta).
\end{eqnarray}
At least formally, we state that $W_{\alpha}(T,Y) \to W_0(T,Y)$ when $\alpha$ tends to $0^+$, and $W_1(T,Y)$ obviously is retrieved as $\alpha \to 1^-$.

To facilitate this study, we limit ourselves to a special class of solutions, namely, one-dimensional (planar) traveling wave solutions $(u,v)$ of the free interface problem which travel at a constant negative velocity $-c$; $c>0$ is to be determined. In the moving frame coordinate, $\xi=x+ct$, the ignition interface is fixed at $\xi=0$, taking advantage of the translation invariance. The trailing front is at an unknown position, $\xi=R$, $0<R<+\infty$.

In the case of zero-order kinetics, we recall that the one-dimensional free interface problem for $(c,R,u(\xi),v(\xi))$ in the moving frame coordinate reads:
\begin{eqnarray}\label{TW_zero}
\left\{
\begin{array}{ll}
 u_{\xi\xi}-cu_{\xi}=0, \quad \Lambda v_{\xi\xi}-cv_{\xi}=0,
 \quad &\xi<0,\\[2mm]
 u_{\xi\xi}-cu_{\xi}=-1, \quad \Lambda v_{\xi\xi}-cv_{\xi}=1, \quad
 &0<\xi<R, \\[2mm]
u_{\xi\xi}-cu_{\xi}=0, \quad
v=0, \quad &\xi>R.
\end{array}
\right.
\end{eqnarray}
The functions $u$ and $v$ are continuously differentiable on the real line such that, according to \eqref{boundary_conditions_0}, $u(-\infty)=0$, $v(-\infty)=1$ and $u(+\infty)=1$. At the interfaces, $u(0)=\theta$ and $v(R)=0$, respectively. It is readily seen that $u(\xi)=1$ for $\xi \geqslant R$; hence, $u_{\xi}(R)=v_{\xi}(R)=0$. In this respect, the trailing interface is a degenerate free boundary in the sense of \cite{BHL00} (in contrast to the ignition interface), which creates further difficulty in the stability analysis (see \cite{BGZ16},\cite{ABLZ20}). Naturally, solving \eqref{TW_zero} is an elementary exercise; it happens that $c=R$ is the unique strictly positive solution of the transcendental equation
\begin{eqnarray}\label{speed_zero}
\displaystyle e^{c^2}=\frac{1}{1-c^2 \theta}, \quad 0<\theta<1.
\end{eqnarray}
(Note that $c$ is independent of $\Lambda$; see Subsection \ref{Lambda=0} below).

For the first-order kinetics, the system for the triplet $(c,u(\xi),v(\xi))$ reads:
\begin{eqnarray}\label{TW_first}
\left\{
\begin{array}{ll}
u_{\xi\xi}-cu_{\xi}=0, \quad \Lambda v_{\xi\xi}-cv_{\xi}=0,
\quad &\xi<0,\\[2mm]
u_{\xi\xi}-cu_{\xi}=-v, \quad \Lambda v_{\xi\xi}-cv_{\xi}=v, \quad
&\xi>0.
\end{array}
\right.
\end{eqnarray}
The functions $u$ and $v$ are continuously differentiable, such that $u(-\infty)=0$, $v(-\infty)=1$ and $u(+\infty)=1$, $v(+\infty)=0$. At the ignition interface, $u(0)=\theta$. This time, the speed $c$ is given explicitly by the formula:
\begin{eqnarray}\label{speed_first}
c= \left(\frac{\theta}{1-\theta} + \Lambda \left(\frac{\theta}{1-\theta}\right)^2 \right)^{-\frac{1}{2}}, \quad 0<\theta<1,\; \Lambda>0.
\end{eqnarray}
Here, there is no trailing interface $R$. More specifically, the trailing interface is such that $R=+\infty$ (see \cite{BSN85}).


In the case of a ``$\alpha$-order kinetic'' \eqref{alpha-order kinetics}, the main difficulty arises from the non-Lipschitz nonlinearity $W_{\alpha}$. Similar issues have recently been addressed in the literature, e.g. a non-Lipschitz modification of the classical Michaelis-Menten law in enzyme kinetics in the analysis of healthy immune system dynamics from the perspective of Finite-Time Stability (FTS) (see \cite{OSV14}).
In this regard, the theory of FTS of continuous, but non-Lipschitz systems has been part of numerous papers particularly associated with optimal control (see, e.g., \cite{BB00},\cite{NHH08}). It is to be seen that the free interface $R$ may also be viewed as a finite time.

In this paper, we prove the following result:
\begin{theorem}\label{FBP_alpha}
Let $0<\alpha<1$, $0<\theta<1$, $\Lambda>0$ be fixed. There exist $c>0$ and $0<R<+\infty$ such that the free interface problem
\begin{eqnarray}
\left\{
\begin{array}{ll}
T_t= \Delta T, \;
Y_t= \Lambda\Delta Y, \quad &x<f(t,y),\\[2mm]
T_t = \Delta T + Y^{\alpha}, \;
Y_t = \Lambda \Delta Y - Y^{\alpha}, \quad &f(t,y)<x<g(t,y), \\[2mm]
T_t = \Delta T, \;
Y(t,x,y)=0, \quad &x>g(t,y),
\end{array}
\right.
\end{eqnarray}
with \eqref{ignition}, \eqref{trailing} and \eqref{boundary_conditions_0}, admits a one-dimensional traveling wave solution $(c,R,u(\xi),v(\xi))$, $\xi=x+ct$, which verifies the free interface problem
\begin{eqnarray}\label{TW_alpha}
\left\{
\begin{array}{ll}
u_{\xi\xi}-cu_{\xi}=0,
\quad &\xi<0,\\[1mm]
\Lambda v_{\xi\xi}-cv_{\xi}=0,
\quad &\xi<0,\\[1mm]

u_{\xi\xi}-cu_{\xi}=-v^{\alpha}, \quad
&0<\xi<R, \\[1mm]
\Lambda v_{\xi\xi}-cv_{\xi}=v^{\alpha}, \quad
&0<\xi<R, \\[1mm]

u_{\xi\xi}-cu_{\xi}=0, \quad &\xi>R,  \\[1mm]
v=0, \quad &\xi>R.
\end{array}
\right.
\end{eqnarray}

The functions $u$ and $v$ are continuously differentiable on the real line such that
\begin{equation}\label{conditions_infty}
u(-\infty)=0, \quad v(-\infty)=1, \quad u(+\infty)=1.
\end{equation}
At the ignition and trailing interfaces, respectively, placed at $\xi=0$ and $\xi=R$, it holds
\begin{equation}\label{interface_conditions}
u(0)=\theta, \quad v(R)=v_\xi(R)=0.
\end{equation}

Moreover, denoting by $R_{\alpha}$ the position of the trailing interface at fixed $\alpha$, we have:
\begin{enumerate}
	\item [\rm (i)] $0<R_{\alpha}<+\infty$, $0<\alpha<1$;
	\item [\rm (ii)] as $\alpha \to 1$, $R_{\alpha} \to +\infty$;
	\item [\rm (iii)] as $\alpha \to 0$, $R_{\alpha} \to R=c$, solution of \eqref{speed_zero}.	
\end{enumerate}
\end{theorem}
 As is seen, the most significant result is that $R$ is finite for $0\leqslant \alpha<1$. As a consequence, the ``cut-off'' exponent for the existence of a finite trailing interface is $\alpha=1$ (see \cite{BSN85} for a comprehensive study of the case $\alpha=1$ and also the case of a $n$-order reaction  rate where $n$ is an integer greater than $1$).

\vskip 2mm
The paper is organized as follows: \\
\indent Firstly, we point out in Section \ref{equivalence2} that the free interface problem \eqref{TW_alpha}-\eqref{interface_conditions} is equivalent to the following one-phase free boundary problem:
\begin{eqnarray}\label{FBP_intro}
\left\{
\begin{array}{ll}
u_{\xi\xi}-cu_{\xi}=-v^{\alpha}, \quad
&0<\xi<R, \\[1mm]
\Lambda v_{\xi\xi}-cv_{\xi}=v^{\alpha}, \quad
&0<\xi<R,
\end{array}
\right.
\end{eqnarray}
with boundary conditions at $\xi=0$
\begin{equation}\label{conditions_zero_intro}
u(0)=\theta,\quad u'(0)=c\theta, \quad v'(0)= -\frac{c}{\Lambda}(1-v(0)),
\end{equation}
and free boundary conditions at $\xi=R$:
\begin{equation}\label{FB_conditions_intro}
v(R)=v'(R)=0, \quad u(R)=1, \quad u'(R)=0.
\end{equation}
The unknowns are: the free boundary $R$; the velocity $c$ which is a kind of eigenvalue of the problem; the functions $u(\xi)$ and $v(\xi)$ which are smooth on the interval $[0,R)$, but whose second derivatives are H\"older continuous at $\xi=R$.
It is immediately apparent that the main quantity in \eqref{FBP_intro}-\eqref{FB_conditions_intro} is the value $v(0)$ that we henceforth take as a parameter
and denote by $v_0\in(0,1)$. In Section \ref{section_scalar}, we formulate a scalar free boundary problem as follows:
\begin{equation}\label{scalar-v_intro}
\begin{cases}
\Lambda v''(\xi)-cv'(\xi)=v^{\alpha}(\xi),\quad 0<\xi<R,\\
v(0)=v_{0},\,
v'(0)=\displaystyle -\frac{c}{\Lambda}(1-v_{0}),\\
v(R)=v'(R)=0,
\end{cases}
\end{equation}
whose study is the main feature of the paper. There are two main results for problem \eqref{scalar-v_intro}: Theorem \ref{theorem_scalar} is about the existence and uniqueness of a solution $(c(v_0),R(v_0),v(v_0;\xi))$; Corollary \ref{corollary_thm_scalar} is about the continuous dependence upon $v_0$. Their proofs, based on topological methods, occupies a large part of this paper, namely Sections \ref{sect-topological-approach} to \ref{proof_thm_scalar}.

The scheme is as follows: we reformulate problem \eqref{scalar-v_intro} as a shooting problem
\begin{equation}\label{IVP-intro}
\begin{cases}
\Lambda v''(\xi)-cv'(\xi)=v^{\alpha}(\xi),\quad \xi>0,\\[1mm]
v(0)=v_{0},\,
v'(0)=\displaystyle -\frac{c}{\Lambda}(1-v_{0}),
\end{cases}
\end{equation}
which is equivalent to finding a trajectory  tending towards the origin (a stable manifold at the origin) of a vector field $X_c$, defined in the quadrant $Q=\{x \geqslant0, y \leqslant0\}$, which means looking at a solution $(x(t),y(t))$ of its differential equation:
\begin{eqnarray}\label{dynamical_syst_intro}
X_c~:
\left\{
\begin{array}{ll}
x'(t)=y(t), \\[2mm]
\Lambda y'(t)=cy(t)+ x^{\alpha}(t),
\end{array}
\right.
\end{eqnarray}
with initial conditions
\begin{equation}\label{IC_intro}
x(0)=v_0, \quad y(0) =\displaystyle -\frac{c}{\Lambda}(1-v_{0}).
\end{equation}
Section \ref{sect-topological-approach} is devoted to the topological study of \eqref{dynamical_syst_intro} for a fixed value of $c$ considered as a parameter, in the first place the phase portrait (see Subsection \ref{subsect-phase-portrait}).  As the vector field $X_c$ is not $C^1$ and the origin is not an hyperbolic singularity when $\alpha <1$, it is impossible to apply Hartman-Grobman Theorem; however, the result stated in Proposition \ref{prop-phase-portrait} can be compared with Hartman-Grobman. The proof of Proposition \ref{prop-phase-portrait} is deferred to Appendix B. In the following Subsections \ref{subsect-local-existence} to \ref{subsect-unicity}, we prove the existence and uniqueness of the stable manifold $y_c(x)$ of \eqref{dynamical_syst_intro}. We use primarily  the  Poincar\'e-Bendixson Theorem (see a brief introduction in Appendix A), which places the origin at the boundary of the domain (refer to \cite{RR} for a comprehensive study).

Section \ref{technical} is devoted to a series of technical lemmata, which enlightens the dependence of the stable manifold $y_c$
upon parameter $c$. In particular,  we give an asymptotic expression of $y_c$ in Subsection \ref{asymptotic}, which eventually provides the optimal H\"older regularity of $v(\xi)$ (see Lemma \ref{holder}).

In Section \ref{proof_thm_scalar}, we proceed to the proof of Theorem \ref{theorem_scalar}, per se. At fixed $v_0$, we prove in Subsection \ref{initial-th_scalar} the existence of a unique $c(v_0)$, such that
$$y_{c(v_0)}(v_0)=-\frac{c(v_0)}{\Lambda}(1-v_0).$$
In Subsection \ref{estimations_time}, we prove that the settling time $R(v_0)$, i.e. the time to proceed from the initial condition \eqref{IC_intro} to the origin on the stable manifold $y_{c(v_0)}$ is finite and, based on the asymptotic expansion of the stable manifold, we estimate $R(v_0)$:
 \begin{equation}\label{eq-bounds-v0_intro}
 \frac{(2(1+\alpha)\Lambda)^{1/2}}{1-\alpha}v_0^\frac{1-\alpha}{2}<R(v_0)< \frac{2\Lambda^{1/2}A(\alpha)}{1-\alpha}\frac{v_0^\frac{1-\alpha}{2}}{1-v_0}.
 \end{equation}
Not surprisingly, this estimate explodes when $\alpha \to 1$ which is coincident with the result given in \cite{BSN85}. Eventually, we return to the free boundary problem \eqref {scalar-v_intro} and prove Corollary \ref{corollary_thm_scalar}.

In Section \ref{well-posedness}, we revisit the free boundary problem \eqref{FBP_intro}-\eqref{FB_conditions_intro}. For $v_0$ in some interval $I\subset (0,1)$ to be determined, we reformulate \eqref{FBP_intro}-\eqref{FB_conditions_intro} as a fixed point problem for the system \eqref{scalar-v_intro}, i.e.,
\begin{equation}\label{scalar-v-bis_introbis}
\begin{cases}
\Lambda v''-cv'=v^{\alpha},\quad 0<\xi<R,\\[1mm]
v(0)=v_{0},\,
v'(0)=\displaystyle -\frac{c}{\Lambda}(1-v_{0}),\\[1mm]
v(R) = v'(R)=0,
\end{cases}
\end{equation}
coupled with an integral equation which reads:
\begin{align}\label{equiv_intro}
\theta+\Lambda v_0=1-c(1-\Lambda)\int_0^Re^{-cs}v(s)\gap ds \quad
\text{or} \quad \theta+v_0=1-(1-\Lambda)\int_0^Re^{-cs}v'(s)\gap ds.
\end{align}
The interval $I$ depends only on $\Lambda$ and $\theta$, especially whether $\Lambda$ is smaller or larger than $1$ (in other words, the Lewis number is larger or smaller than unity). We prove the existence of a fixed point in Theorems \ref{thm_fixpt_Phi} and \ref{thm_fixpt_Psi}, taking advantage of the continuity of $(c(v_0), R(v_0),v(v_0;\xi))$ w.r.t. $v_0$.

The last section of the paper is devoted to a gamut of special cases. In particular, we consider the situation of a solid combustion where the Lewis number is $+\infty$, i.e. $\Lambda=0$. The latter case is mathematically relevant, because it allows explicit computations that yield the uniqueness of the solution to the free boundary problem. Finally, we consider the limit cases $\alpha \to 1$ and $\alpha \to 0$, which complete the proof of Theorem \ref{FBP_alpha}.

As already mentioned, the paper has two appendices: Appendix A regarding the Poincar\'e-Bendixson Theorem and Appendix B devoted to the proof of Proposition \ref{prop-phase-portrait}.

\section{Equivalence with a one-phase free boundary problem}\label{equivalence2}
For given $0<\alpha<1$ and $\Lambda>0$, let us examine the free interface problem \eqref{TW_alpha}-\eqref{interface_conditions}: it is easy to integrate \eqref{TW_alpha} for $\xi<0$, therefore
\begin{equation*}\label{xi_negative}
\displaystyle u(\xi)=\theta e^{c\xi}, \quad
 v(\xi)=1-(1-v(0))e^{\frac{c}{\Lambda}\xi}, \quad \xi<0.
\end{equation*}
Because we look for $u$ and $v$ continuously differentiable, it comes
\begin{equation*}\label{conditions_zero}
u(0)=\theta,\quad u'(0)=c\theta, \quad v'(0)= -\frac{c}{\Lambda}(1-v(0)),
\end{equation*}
and for the same reason, it follows that $u(\xi)=1$ for all $\xi \geqslant R$. This suggests the following simplification.
\begin{proposition}\label{proposition_simplification}
	Let $0<\alpha<1$ and $\Lambda>0$ be fixed. The free interface problem \eqref{TW_alpha}-\eqref{interface_conditions} is equivalent to the one-phase free boundary problem: find $c>0$, a free boundary $0<R<+\infty$, $u$ and $v$ in $C^{\infty}([0,R))\cap C^1([0,R])$, such that $(c,R,u(\xi),v(\xi))$  verifies
	\begin{eqnarray}\label{free_boundary_problem}
	\left\{
	\begin{array}{ll}
	u_{\xi\xi}-cu_{\xi}=-v^{\alpha}, \quad
	&0<\xi<R, \\[1mm]
	\Lambda v_{\xi\xi}-cv_{\xi}=v^{\alpha}, \quad
	&0<\xi<R,
	\end{array}
	\right.
	\end{eqnarray}
	with the following boundary conditions at $\xi=0$

	\begin{equation}\label{conditions_zero}
	u(0)=\theta,\quad u'(0)=c\theta, \quad 0<v(0)<1, \quad v'(0)= -\frac{c}{\Lambda}(1-v(0)),
	\end{equation}
	and the free boundary conditions at $\xi=R$
	\begin{equation}\label{free_boundary_conditions}
	v(R)=v'(R)=0, \quad u(R)=1, \quad u'(R)=0.
	\end{equation}
	Moreover, the following integral relation holds
	\begin{equation}\label{integral_v}
	c=\int_0^R v^{\alpha}(\xi)\gap d\xi.
	\end{equation}
\end{proposition}
\begin{proof}
	(i)  Assume there exists $c>0$, $R>0$, $u\in C^1(\R)$, $v\in C^1(\R)$ such that $(c,R,u(\xi),v(\xi))$ verifies the free interface problem \eqref{TW_alpha}-\eqref{interface_conditions}. It is clear that $u$ is smooth on $(-\infty,0)\cup(0,R)\cup(R,+\infty)$, $v$ is smooth on $(-\infty,0)\cup(0,R)$ (recall that $v$ is identically zero on $[R,+\infty)$). Solving \eqref{TW_alpha} for $\xi<0$ and taking \eqref{conditions_infty} and \eqref{interface_conditions} into account, we obtain
	\begin{equation}\label{xi_negative}
	\begin{cases}
	\displaystyle u(\xi)=\theta e^{c\xi}, \quad &\xi<0,\\[1mm]
	v(\xi)=1-(1-v(0))e^{\frac{c}{\Lambda}\xi}, \quad &\xi<0.
	\end{cases}
	\end{equation}
	Thus, \eqref{conditions_zero} follows immediately.
	Next, solving $u_{\xi\xi}-cu_{\xi}=0$ for $\xi>R$ with $u(+\infty)=1$ yields $u(\xi)=1$ for $\xi \geqslant R$. Because $u$ and $v$ are continuously differentiable, the free boundary conditions \eqref{free_boundary_conditions} are verified. Finally,
 $(c,R,u|_{[0,R]},v|_{[0,R]})$  satisfies the free boundary problem \eqref{free_boundary_problem}-\eqref{free_boundary_conditions}. Note that in \eqref{free_boundary_conditions} the conditions $u(R)=1$ and $u'(R)=0$ are equivalent.
	
(ii) Now, assume there exist $c>0$, $R>0$, $u$ and $v$ in $C^{\infty}([0,R))\cap C^1([0,R])$ such that $(c,R,u(\xi),v(\xi))$ is a solution to \eqref{free_boundary_problem}-\eqref{free_boundary_conditions}. Then, we extend  $u$ and $v$, respectively, by $\tilde u$ and $\tilde v$  to the whole line as follows: we define  $\tilde u$ and $\tilde v$ for $\xi<0$ via \eqref{xi_negative}; for $\xi>R$ we set $\tilde{u}\equiv 1$, $\tilde{v}\equiv 0$. By construction, $\tilde u$ and $\tilde v$ are in $C^1(\R)$. Finally, it is easy to check that $(c,R,\tilde{u}(\xi),\tilde{v}(\xi))$ is a solution of  \eqref{TW_alpha}-\eqref{interface_conditions}.

(iii)  We easily obtain formula \eqref{integral_v} by integrating the equation $\Lambda v_{\xi\xi}-cv_{\xi}=v^{\alpha}$ between $0$ and $R$, and taking \eqref{conditions_zero} and \eqref{free_boundary_conditions} into account. Note that $c=R$ when  $\alpha=0$ (see \eqref{TW_zero}).
\end{proof}

\section{The scalar free boundary problem}\label{section_scalar}
It transpires from Proposition \ref{proposition_simplification} that the main parameter of the free boundary problem \eqref{free_boundary_problem}-\eqref{free_boundary_conditions} is $v(0)$. We first focus on the scalar problem for $v$ with free boundary $R$, assuming $v(0)=v_0$ given in the interval $(0,1)$.
Throughout the next sections, we intend to prove the following theorem:

\begin{theorem}\label{theorem_scalar}
Let $0<\alpha<1$, $\Lambda>0$ be fixed. Then, for each $v_0\in(0,1)$, there exists a unique solution $(c(v_0),R(v_0),v(v_0;\xi))$ to the one-phase free boundary problem
 \begin{equation}\label{scalar-v}
 \begin{cases}
 \Lambda v''(\xi)-cv'(\xi)=v^{\alpha}(\xi),\quad 0<\xi<R,\\
v(0)=v_{0},\,
v'(0)=\displaystyle -\frac{c}{\Lambda}(1-v_{0}),\\
v(R)=v'(R)=0,
\end{cases}
\end{equation}
such that $c(v_0)>0$, $0<R(v_0)<+\infty$, $v\in C^{\infty}([0,R))\cap C^{2+[\beta],\beta-[\beta]}([0,R])$, $\beta=\mfrac{2\alpha}{1-\alpha}$. Moreover,
\begin{equation}
v(\xi)>0,\quad v'(\xi)<0, \quad v''(\xi)>0, \quad \forall \xi \in [0,R).
\end{equation}
\end{theorem}

\begin{corollary}\label{corollary_thm_scalar}
Let $0<\alpha<1$, $\Lambda>0$ be fixed. Set $I=[a,b]$, $0<a<b<1$. For $v_0 \in I$, let $(c(v_0),R(v_0),v(v_0;\xi))$ denote the solution of system \eqref{scalar-v}. \\
(i) There exists a finite $R_{max}$ depending only upon $\alpha$, $\Lambda$, $a$ and $b$ such that $0<R(v_0)\leqslant R_{max}$ for all $v_0 \in I$. \\
(ii) The mapping $v_0 \mapsto (c(v_0),R(v_0),\widetilde{v}(v_0))$ from $I$ to $(\R^+)^2 \times C^1([0,R_{max}])$ is continuous, where $\widetilde{v}(v_0;\xi)$ is the extension by $0$ of the function $v(v_0;\xi)$ to the interval $[0,R_{max}]$.
\end{corollary}

\subsection{Formal a priori estimates for $c(v_0)$}\label{formal_estimate}
At the outset, it is easy to establish a priori estimates for $c>0$, assuming that $(c(v_0),R(v_0),v(v_0;\xi))$ is a solution to \eqref{scalar-v}.
Multiplying formally the equation by $v'$ and integrating from $0$ to $R$, using the boundary and free boundary conditions in \eqref{scalar-v}, we arrive at
\begin{equation}\label{upper-c1}
\frac{1}{2}\frac{c^2}{\Lambda}(v_0-1)^2+c\int_0^R (v')^2\, d\xi=\frac{1}{1+\alpha}v_0^{1+\alpha}.
\end{equation}
Thus, we immediately achieve the upper bound $c<c_+$ with
\begin{equation}\label{upper-c2}
c^2_+=\frac{2\Lambda}{1+\alpha}\frac{v_{0}^{1+\alpha}}{(v_{0}-1)^{2}}.
\end{equation}
Likewise, we multiply the equation by $v$.
Integrating
from $0$ to $R$ and using again the boundary and free boundary conditions in \eqref{scalar-v}, we obtain
\begin{equation}
-cv_0(v_0-1)-\Lambda\int_0^R (v')^2\, d\xi + \frac{c}{2}v^2_0=\int_0^R v^{1+\alpha}\, d\xi,
\end{equation}
which together with \eqref{upper-c1} gives the lower bound $c>c_-$ with
\begin{equation}\label{lower-c2}
c^2_-=\frac{2\Lambda}{1+\alpha}v_0^{1+\alpha}.
\end{equation}
Summarizing, we have the following a priori estimates:
	\begin{equation}\label{estimate_c}
	0<c_-(v_0)<c(v _0)<c_+(v_0).
	\end{equation}
These bounds are proved via a different method (see Subsection \ref{estimate_c(v_0)} below).
\subsection{Proof of uniqueness}
Here, we give a direct proof of uniqueness in Theorem \ref{theorem_scalar}. We find an alternate proof in Section \ref{sect-topological-approach} (see Lemma \ref{lem-unicity}).
\begin{lemma}
Let $v_0 \in (0,1)$ be fixed. There exists a unique solution $(c(v_0),R(v_0),v(v_0;\xi))$ to system \eqref{scalar-v}.
\end{lemma}
\begin{proof}
	Assume that $(c_i(v_0),R_i(v_0),v_i(v_0;\xi))$, $i=1,2$, are two different solutions to \eqref{scalar-v}, namely
	\begin{equation}\label{scalar-vi}
	\begin{cases}
	\Lambda v_i''(\xi)-c_iv_i'(\xi)=v_i^{\alpha}(\xi),\quad 0<\xi<R_i,\\
	v_i(0)=v_{0},\,
	v_i'(0)=\displaystyle -\frac{c_i}{\Lambda}(1-v_{0}),\\
	v_i(R_i)=v_i'(R_i)=0.
	\end{cases}
	\end{equation}
	Without loss of generality, we may assume that $c_2>c_1$. \\
	(i) First, we prove that
	\begin{equation}\label{inequ2}
	R_2<R_1\quad \text{and}\quad v_2(\xi)< v_1(\xi),\, \forall \xi\in(0,R_2).
	\end{equation}
	Set $\varphi=v_1-v_2$, $\varphi(0)=0$. As $|v'_2(0)|>|v'_1(0)|$, one has
	$\varphi'(0)>0$, thus $\varphi>0$. Assume by contradiction that the two solutions intersect: let $\xi_0\in (0,\bar{R})$, $\bar{R}= \min\{R_1,R_2\}$, be the first point such that $\varphi(\xi_0)=0$, i.e.
	$\varphi(\xi)>0, \forall \xi\in(0,\xi_0)$.
	Noticing that $\varphi\in C^2([0,\bar{R}])$, by Rolle's theorem there exists a $\bar\xi\in(0,\xi_0)$ such that $\varphi'(\bar\xi)=0$.
	From \eqref{scalar-vi}, one has
	\begin{equation}\label{varphi}
	\Lambda \varphi''-c_{1}\varphi'+v'_{2}(c_{2}-c_{1})=v_{1}^{\alpha}-v_{2}^{\alpha}.
	\end{equation}
	Integrating \eqref{varphi} from $0$ to $\bar\xi$ and using $\varphi'(\bar \xi)=\varphi(0)=0$, we obtain
	\begin{equation}\label{integral_RHS}
	-\Lambda\varphi'(0)-c_1\varphi(\bar\xi)+(c_2-c_1)(v_2(\bar\xi)-v_2(0))=\int_0^{\bar\xi} (v_1^{\alpha}-v_2^{\alpha})\, d\xi.
	\end{equation}
	Thanks to $\varphi'(0)>0$, $\varphi(x)>0$ and $v_2'(\xi)<0,\forall \xi\in(0,\bar\xi)$,
	it is found that the left-hand side of \eqref{integral_RHS} is negative while the right-hand side is positive, thus, a contradiction.\\
	(ii) Next, according to formula \eqref{integral_v}, it holds
	\begin{equation}
	c_{1}=\int_{0}^{R_{1}}v_{1}^{\alpha}d\xi\quad \text{and}\quad c_{2}=\int_{0}^{R_{2}}v_{2}^{\alpha}d\xi,
	\end{equation}
	hence from \eqref{inequ2}
	$$c_{2}=\int_{0}^{R_{2}}v_{2}^{\alpha}d\xi < \int_{0}^{R_{2}}v_{1}^{\alpha}d\xi < \int_{0}^{R_{1}}v_{1}^{\alpha}d \xi=c_{1},$$
	which contradicts with the assumption $c_2>c_1$; the uniqueness is thus proved.
\end{proof}

\subsection{Scheme of the proof of Theorem \ref{theorem_scalar}}\label{scheme}
Let $0<\alpha<1$, $\Lambda>0$, $0<v_0<1$ be fixed. For $c>0$, we consider the system
\begin{equation}\label{IVP}
\begin{cases}
\Lambda v''(\xi)-cv'(\xi)=v^{\alpha}(\xi),\quad \xi>0,\\[1mm]
v(0)=v_{0},\,
v'(0)=\displaystyle -\frac{c}{\Lambda}(1-v_{0}),
\end{cases}
\end{equation}
and we look for solution $v$ such that $v(\xi)>0$, $v'(\xi)<0$.
Now, we set:
\begin{equation}\label{notation_syst_dyn}
t=\xi, \quad x(t)=v(\xi), \quad y(t)=v'(\xi).
\end{equation}
Therefore, problem \eqref{IVP} is equivalent to the following differential system in the quadrant $Q=\{x \geqslant0, y \leqslant0\}$
\begin{eqnarray}\label{dynamical_syst}
X_c~:
\left\{
\begin{array}{ll}
x'(t)=y(t), \\[2mm]
\Lambda y'(t)=cy(t)+ x^{\alpha}(t),
\end{array}
\right.
\end{eqnarray}
with initial conditions
\begin{equation}\label{IC}
x(0)=v_0, \quad y(0) =\displaystyle -\frac{c}{\Lambda}(1-v_{0}).
\end{equation}
The origin $O=(0,0)$ is the unique critical point of the vector field $X_c$ given by (\ref{dynamical_syst}). Clearly, Theorem \ref{theorem_scalar} demonstrates the existence of $c>0$ and a trajectory $(x(t),y(t))$ in $Q$, defined on $[0,R)$ with initial conditions $(x(0),y(0))$ given by  \eqref{IC}, which tends towards $O$ for $t\rightarrow R_-$. (It also can be said that $O$ is the $\omega$-limit set of the trajectory; $(x(t),y(t))$ extends continuously on $[0,R]$ by $(x(R),y(R))=(0,0)$.

We prove Theorem \ref{theorem_scalar} by using primarily  the  Poincar\'e-Bendixson theorem (see Appendix A), and, from time to time, by using some  other simple qualitative arguments. As mentioned above, the main difficulty is that the vector field is not of class ${C}^1$ along the axis $\{x=0\}$ when $0<\alpha<1.$ This is overcome by putting the origin at the boundary of the domain.

Section \ref{sect-topological-approach} is devoted to the topological study of the differential equation \eqref{dynamical_syst}, for a fixed value of $c$, especially the phase portrait and the existence of a stable manifold. The dependence on $c$ w.r.t.~$v_0$ is studied in Section \ref{technical}. In Section \ref{proof_thm_scalar}, we approach the proof of Theorem \ref{theorem_scalar}, per se. The critical result (see Lemma \ref{lem-time-T}) is that the settling time $R(v_0)$ to proceed from the initial condition to the origin is finite. Eventually, we return to the notation of the free boundary problem \eqref {scalar-v} and prove Corollary \ref{corollary_thm_scalar}.

\section{ A topological approach ($c$ fixed)}\label{sect-topological-approach}

In this section, the parameter $c \geqslant 0$ is fixed; $\alpha$ is such that $0\leqslant \alpha\leqslant 1$, and we assume $v_0>0$ (note that these hypotheses go beyond the physical framework). We intend to study the flow of the vector field $X_c$ with differential equation \eqref{dynamical_syst}, that reads:
\begin{eqnarray}\label{dynamical_syst2}
\left\{
\begin{array}{ll}
x'=y, \\[2mm]
 y'=\displaystyle\frac{1}{\Lambda}(cy+ x^{\alpha}),
\end{array}
\right.
\end{eqnarray}
in a topological way.
We recall that the vector field $X_c$ is a first order differential operator $y\frac{\partial}{\partial x}+\frac{1}{\Lambda}(cy+ x^{\alpha})\frac{\partial}{\partial y}$. We write $X_c\cdot f$ its  action  on a function $f$:
\begin{equation*}
X_c\cdot f(x,y)=y\frac{\partial f}{\partial x}(x,y)+\frac{1}{\Lambda}(cy+ x^{\alpha})\frac{\partial f}{\partial y}(x,y).
\end{equation*}
This vector field is considered in the quadrant $Q=\{x\geqslant 0, y\leqslant 0\}.$ We denote its flow in $Q$ by $\varphi_c(t,m)$.
(For convenience, some information about  vector fields is given in Appendix A).

In the first subsection, we state a general topological result: the phase portrait of the vector field $X_c$, extended to the whole plane,  is independent of $c$, $\alpha$ and $\lambda$. The proof of this result (given in Appendix B), is based in a crucial way on the existence and uniqueness of the stable manifold. The latter result is established below in Subsections \ref{subsect-local-existence}, \ref{subsect-global-existence} and \ref{subsect-unicity}.

\vskip10pt

\subsection{Phase portrait}\label{subsect-phase-portrait}
We extend the vector field $X_c$ to the whole plane by defining:
\begin{equation}\label{eq-X-extended}
X^E_c(x,y)=y\frac{\partial}{\partial x}+\frac{1}{\Lambda}(cy+{\rm sign}(x)|x|^\alpha)\frac{\partial}{\partial y}.
\end{equation}
\begin{remark}
To have a non-ambiguous definition, even for $\alpha=0$, we consider that
${\rm sign}(0)=0$ (and $\pm 1$ otherwise). When $\alpha=0$, $ X^E_c$ is well defined, but discontinuous along the $Oy$-axis,  outside the origin, which is a singular point.
\end{remark}
As $X^E_c|_Q=X_c,$ this new field is an extension of $X_c.$
It is observed  that
\begin{equation*}
X^E_c(-x,-y)=-X^E_c(x,y).
\end{equation*}
It follows that the orbits for $\{x\leqslant 0\}$ are obtained from the orbits for $\{x\geqslant 0\}$ by using the symmetry $(x,y)\mapsto (-x,-y)$ and next changing the time orientation (see Figure \ref{fig-Phase-Portrait}).  The trajectories of $X^E_c$ cross in a regular way the $Oy$-axis at points different from the origin, and the origin is an equilibrium point.

In this subsection we present a topological description for the vector field $X^E_c$, based on the following definition (see \cite{A},\cite{H},\cite{HS},\cite{DLA}):
\begin{definition}\label{def-phase-portrait}
Two vector fields $X$ and $Y$ defined on the same manifold $M$, are said topologically equivalent if there exists an homeomorphism $H$ of $M$, which sends the orbits of $X$ onto the orbits of $Y$ (preserving time-orientation). To be topologically equivalent is clearly an equivalence relation (in the space of vector fields on $M$). An equivalence class for the topological equivalence  is called a phase portrait.
\end{definition}
That is to say, two vector fields are topologically equivalent if they have the same phase portrait.
We have the following global result whose proof is given in Appendix B:
\begin{proposition}\label{prop-phase-portrait}
The phase  portrait of $X^E_c$ is independent of $c,\alpha$  and even of $\Lambda>0.$ In other words, the general vector field $X^E_c$ is topologically equivalent on $\R^2$ to the linear saddle type vector field $X_0=y\frac{\partial}{\partial x}+x\frac{\partial}{\partial y}$ corresponding to $\Lambda=1,\alpha=1,c=0$.
\end{proposition}
\vskip5pt
\begin{figure}[htp]
\begin{center}
   \includegraphics[scale=0.7]{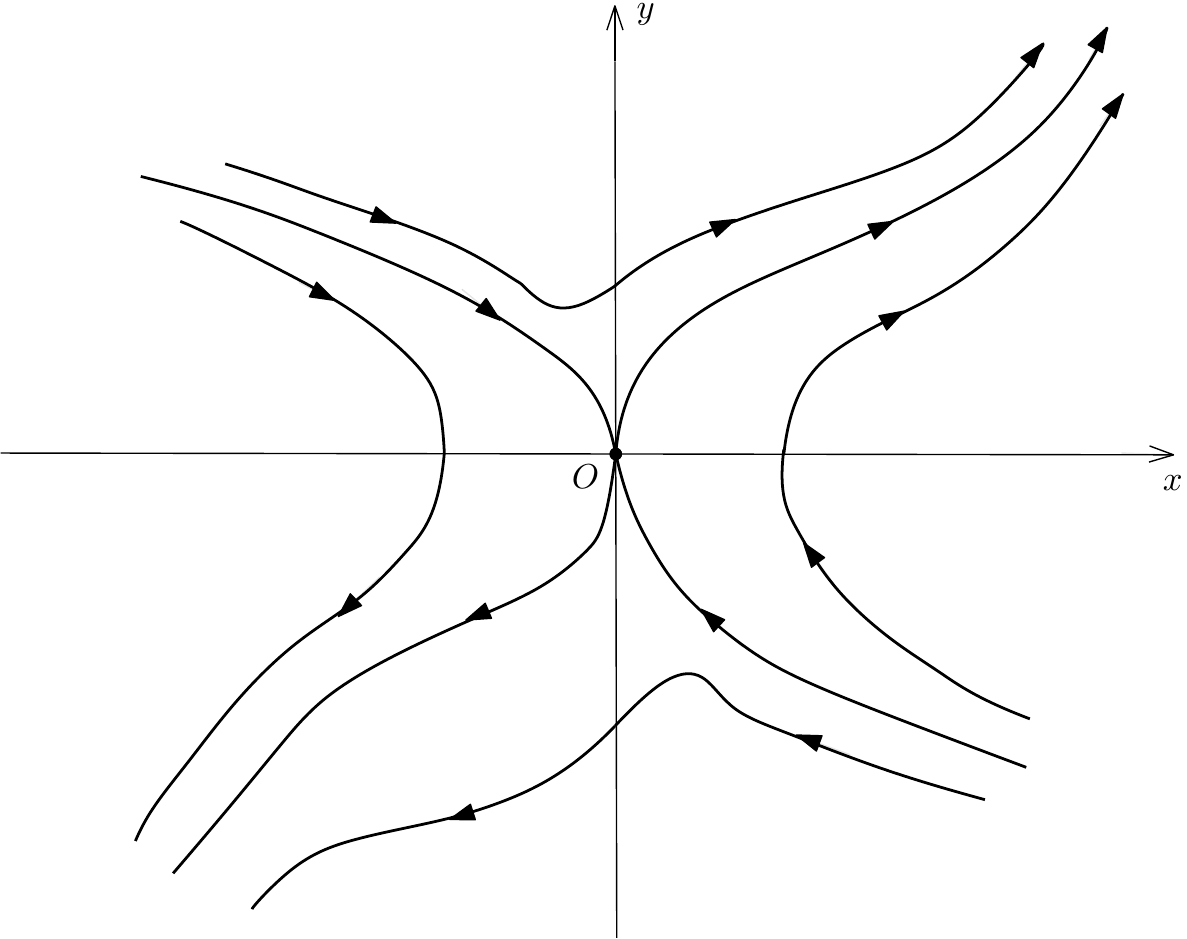}\\
  \caption{Phase portrait of $X^E_c$ for $\alpha <1$. It fulfills some differentiable features of the flow such as the contact of orbits at the origin.}
  \label{fig-Phase-Portrait}
  \end{center}
\end{figure}
\vskip5pt
The result stated in Proposition \ref{prop-phase-portrait} can be compared with the Hartman-Grobman Theorem (see, e.g.\cite{S}). Of course, as the vector field  $X^E_c$ is not $C^1$, and the origin is not an hyperbolic singularity when $\alpha < 1$, it is impossible to apply the Hartman-Grobman Theorem.\footnote{Moreover, the time to proceed from one point on the stable manifold to the origin is infinite for $X_0$ but finite  for $X^E_c$ when $\alpha<1$, see Subsection \ref{estimations_time}.} In Appendix B we construct directly ``by hands'' the homeomorphism $H$ of equivalence, beginning with the construction on quadrant $Q$ on which we consider $X_c$.
Next, $H$ is extended in the whole plane in a rather direct way. The fact that $X_c$ has a unique stable manifold in $Q$ is crucial for this construction.

\subsection{Existence of a local stable manifold}\label{subsect-local-existence}
For $c=0,$ the vector field $X_0$ is Hamiltonian with Hamiltonian function $\displaystyle H(x,y)= \frac{1}{2}y^2-\frac{1}{\Lambda(1+\alpha)}x^{1+\alpha}.$ This vector field has stable manifold  in $Q$ at the origin $O:$
\begin{equation}\label{eq-L0}
L_0 := \Big\{y=y_0(x)=-\Big(\frac{2}{\Lambda(1+\alpha)}\Big)^{1/2}x^{\frac{1+\alpha}{2}}\Big\}.
\end{equation}
For any {$c\geqslant 0$}, we have that:
\begin{equation*}
X_c\cdot H(x,y)=-\frac{1}{\Lambda}yx^\alpha+\frac{1}{\Lambda}(cy+ x^{\alpha})y=\frac{c}{\Lambda}y^2.
\end{equation*}
This implies that,  for $c>0$, the vector field $X_c$ is transverse to $L_0$ and directed downwards all along $L_0,$ outside $O$.

Now, for any $v_0>0,$ we consider in $\R^2$ the region
$\mathcal{T}_{v_0}=\{0\leqslant x \leqslant v_0,\, y_0(x)\leqslant y\leqslant 0\}$.
This compact region is a curved triangle that we call a {\it trapping triangle} for a reason that will become apparent. This triangle has the following three corners: $O,\,A_{v_0}=(v_0,0)$ and $B_{v_0}=(v_0, y_0(v_0))$; and three sides denoted as follows: $[OA_{v_0}],[A_{v_0}B_{v_0}]$ and
$[OB_{v_0}]$. The vector field $X_c$  is transverse  and has an upward direction along $(OA_{v_0}]=[OA_{v_0}]\setminus\{O\}.$ It is transverse and has a left direction along the open interval $(A_{v_0}B_{v_0})=[A_{v_0}B_{v_0}]\setminus \{A_{v_0}\}\cup\{B_{v_0}\}.$ As already mentioned, $X_c$ is transverse and has a downward direction along $(OB_{v_0}]=[OB_{v_0}]\setminus \{O\}$ (see Figure \ref{fig-loc-sta-man}).
\vskip5pt
\begin{figure}[htp]
\begin{center}
   \includegraphics[scale=0.7]{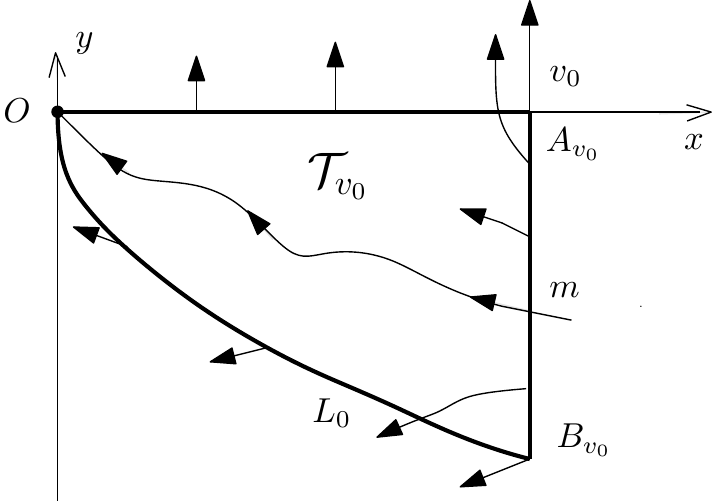}\\
  \caption{Local stable manifold}\label{fig-loc-sta-man}
  \end{center}
\end{figure}
\vskip5pt
The following lemma proves the existence of a local stable manifold contained inside $\mathcal{T}_{v_0}$:
\begin{lemma}\label{lem-exist-local-s-m}
There exist points $m\in (A_{v_0}B_{v_0})$ whose trajectory $\varphi_c(t,m)$ remains in $T_{v_0}$ for all times and tends toward $O$ when $t\rightarrow \tau^+(m)$. Let $S$ be the set of points in $ (A_{v_0}B_{v_0})$ whose trajectory tends toward $O.$ Trajectories starting at  points  $ (A_{v_0}B_{v_0})\setminus S$ cut the side $(OA_{v_0}]$ or the side $(OB_{v_0}]$ after a finite time.
\end{lemma}
\begin{proof}
We consider  any point $m\in (A_{v_0}B_{v_0})$  and look at its trajectory in
$\mathcal{T}_{v_0}$. More precisely, to remain in the interior $\Omega_{v_0}$ of
$\mathcal{T}_{v_0}$, we start at a point $\varphi_c(T_m,m)$ for $T_m>0$ small enough, and we consider the positive  trajectory of this point inside the open set $\Omega_{v_0}.$ The trajectory is defined on a time  interval $[T_m,\tau^+(m))$.

Using Corollary \ref{P-B-corollary}, we know that   there exists a sequence $(t_n)\rightarrow \tau^+(m)$ such that the sequence of points $\varphi_c(t_n,m)$ tends toward a point $p\in\partial \mathcal{T}_{v_0}$. Clearly, $p$ cannot belong to $(A_{v_0}B_{v_0})$, as the vector field is repulsive along this side. If $p\in (OA_{v_0}],$ then it is easy to see, using a flow box centered at $p$, that the trajectory arrives at $p$ at the finite time $\tau^+(m)$   and
crosses the axis $Ox$ at the regular point $p=\varphi(\tau^+(m),m)$.  The same remark is valid if $p\in (OB_{v_0}]$. In this case, the trajectory reaches a regular point on $L_0\setminus \{O\}$ in a finite time.

The set $\mathcal{O}_{up}$ of points $m$ for which the trajectory arrives on $(OA_{v_0}]$ is open in $(A_{v_0}B_{v_0})$, as the transversality of  this trajectory at each end to  $[A_{v_0}B_{v_0}]$ and $Ox$ respectively is an open condition. Moreover, $\mathcal{O}_{up}$ is non-empty, as this set contains points near $A_{v_0}$.  In a similar way, it is seen that the set  $\mathcal{O}_{down}$ of points $m$ for which the trajectory arrives on  $(OB_{v_0}]$ is a non-empty open subset of $(A_{v_0}B_{v_0})$. However, as $(A_{v_0}B_{v_0})$ is connected, it cannot be the union of the two non-empty and disjoint open sets
$\mathcal{O}_{up}$ and $\mathcal{O}_{down}$. This means that there are points $m\in (A_{v_0}B_{v_0})$ for which the limit point $p$ is the origin $O$.

Then, consider such a point $m\in (A_{v_0}B_{v_0})$  with the property that there exists a time sequence $(t_n)\rightarrow \tau^+(m)$ such that $\varphi_c(t_n,m)\rightarrow O$.  This means that for each $\varepsilon>0$ small enough, there exists a
$n(\varepsilon)$ such that $\varphi_c(t_{n(\varepsilon)},m)\in \mathcal{T}_\varepsilon$, where $\mathcal{T}_\varepsilon$ is the small triangle obtained by replacing $v_0$ by $\varepsilon$.  However, as the vertical side of $\mathcal{T}_\varepsilon$ contained in $\{x=\varepsilon\}$ is repulsive (the vector field is directed toward the left) and as the trajectory of $m$ remains in $\mathcal{T}_{v_0}$ by definition (we consider the trajectories only into
$\Omega_{v_0}$), the trajectory of $m$ is trapped in $\mathcal{T}_\varepsilon$; hence, the point $\varphi_c(t,m)$ belongs to  $\mathcal{T}_\varepsilon$ for all
$t\geqslant t_{n(\varepsilon)}.$ As the diameter of
$\mathcal{T}_\varepsilon$ tends to zero with $\varepsilon$, it is implied that
$\varphi_c(t,m)\rightarrow O$ for $t\rightarrow \tau^+(m)$. Therefore, the trajectory of $m$ is a local stable manifold contained in $\mathcal{T}_{v_0}$.
\end{proof}
In the next subsections we use more general trapping triangles:
{\begin{definition}\label{def-trap-triangle}
A trapping triangle $\mathcal{T}_{v_0}$ for the vector field $X_c$ is a curved triangle in $Q$ with a corner at $O$ and two other corners $A_{v_0},B_{v_0}$ on a vertical line $\{x=v_0\}$, for some $v_0>0$.  The side $[OA_{v_0}]$ is contained in  a graph $ L^u  :=\{y=y^u(x)\}$,  the side $[OB_{v_0}]$ in a graph $L^d :=\{y=y^d(x)\}$ and the side $[A_{v_0}B_{v_0}]$ is a vertical interval in  $\{x=v_0\}$.
We assume that $L^u$ is located above $L_d$, i.e. $y^d(x)<y^u(x)\leqslant 0$ for all $x>0$ (and by hypothesis: $y^d(0)=y^u(0)=0$). We also assume that  $X_c$ is transverse in the upward direction all along  $(OA_{v_0}]=[OA_{v_0}]\setminus \{O\}$ and is transverse in the downward direction all along $(OB_{v_0}]$. It is clear that $X_c$ is transverse to $(A_{v_0},B_{v_0}]$ in the left direction (see Figure \ref{fig-trap-tri}).
\end{definition}
\vskip5pt
\begin{figure}[htp]
\begin{center}
   \includegraphics[scale=0.7]{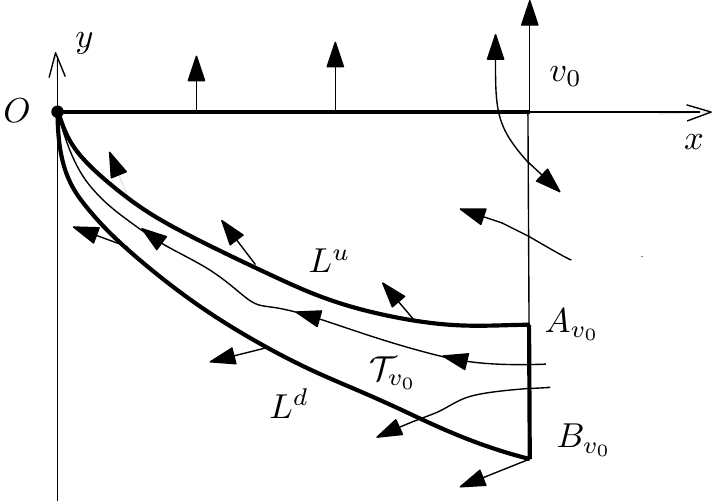}\\
  \caption{\it Trapping triangle}\label{fig-trap-tri}
  \end{center}
\end{figure}
\vskip5pt
The triangle used in the proof of Lemma \ref{lem-exist-local-s-m} is obviously a trapping triangle.   Any trapping triangle  verifies the statement of  Lemma \ref{lem-exist-local-s-m}. As the proof is exactly the same, we state the result  without giving a new proof:

\begin{lemma}\label{lem-general-trap-triangle}
Let  $\mathcal{T}_{v_0}$ be a trapping triangle as in Definition \ref{def-trap-triangle}.
There exist  points $m\in (A_{v_0}B_{v_0})$ whose trajectory $\varphi_c(t,m)$ remains into $\mathcal{T}_{v_0}$ for all times and tends toward $O$ for $t\rightarrow \tau^+(m)$. Trajectories starting at other points of $ (A_{v_0}B_{v_0})$ cut the side $(OA_{v_0}]$ or the side $(OB_{v_0}]$ after a finite time.
\end{lemma}
\subsection{Existence of global stable manifold}\label{subsect-global-existence}
In this subsection we see how to extend a local stable manifold found in Subsection \ref{subsect-local-existence} into a global one. Such a global stable manifold is a trajectory contained in the quadrant $Q$ and converging toward the origin. We use again Corollary \ref{P-B-corollary}, but now in an easier way, as the vector field $X_c$ is analytic in a neighborhood of the domain under consideration.

\begin{lemma}\label{lem-exist-global-s-m}
There exists a global stable manifold at the origin, given by a trajectory of $X_c$ in $Q$
converging toward the origin in the positive direction of time. The orbit defined by this trajectory (its geometrical image) is the graph of a function $y=y(x)$ with  $x\in (0,+\infty)$.
This function is analytic for $x>0$ and extends continuously at $0$ with the value $y(0)=0$. Its graph is located above $L_0,$ more precisely, $y_0(x)<y(x)<0$ for all $x>0$ ($y=y_0(x)$ is the equation of the curve $L_0$ defined by (\ref{eq-L0})).
\end{lemma}
\begin{proof}
We start with a local stable manifold found in Lemma \ref{lem-exist-local-s-m}, which is defined by a graph $\{y=y(x)\}$  for $0\leqslant x\leqslant v_0$ where $v_0>0$. As $v_0$ can take any positive value,  the idea might be to make $v_0$ go to $+\infty$.  But the problem is that local stable manifolds may depend on $v_0$. Therefore, we proceed differently to obtain the desired global manifold, using again Corollary \ref{P-B-corollary}.

A value $v_0>0$ being chosen, we consider the sequence $v_n=nv_0$.   For each $n\in \N$,  we consider the curved  square $S_n$ defined
by $S_n=\{(x,y)\  v_n\leqslant x\leqslant v_{n+1},\  y_0(x)\leqslant y\leqslant 0\}$. $S_n$ has four corners: $A_n=(v_n,0), \   B_n=(v_n,y_0(v_n)), \  B_{n+1}=(v_{n+1},y_0(v_{n+1}))$ and $A_{n+1}=(v_{n+1},0)$; and four sides: the horizontal linear segment $[A_nA_{n+1}]$, the two vertical linear segments $[A_nB_n],[A_{n+1} B_{n+1}]$ and the curved segment $[B_n B_{n+1}]$ located on the curve $L_0$. The vector field $X_c$  is analytic and without singularities in a neighborhood of $S_n$.  Moreover, the vector field is transverse along each side, with an entering direction along the side $[A_{n+1} B_{n+1}]$ and a leaving  direction along the three over sides (see  Figure \ref{fig-glob-sta-man}).

Now, applying Corollary \ref{P-B-corollary} in a neighborhood of $S_n$ to the vector field $-X_c,$ we have that if $m_n$ is a point in the interior of the segment  $[A_nB_n]$ (i.e., $m_n=(v_n,y_n)$ with $y_0(v_n)<y_n<0)$, then there exists a $\bar t_n<0$ such that $\varphi(\bar t_n,m_n)$ belongs to the interior of the segment $[A_{n+1} B_{n+1}]$.

  Let $m_0=(v_0,y_0)$ be the  endpoint
of a local stable manifold founded by Lemma \ref{lem-exist-local-s-m} in the trapping triangle $\mathcal{T}_{v_0}.$ We have that $m_0$ belongs to the interior of $[A_0B_0]$. Now, using the above
result as  an induction step in the construction of a global stable manifold, we find a decreasing sequence of times $t_0=0>t_1>t_2>\cdots >t_n\cdots,$ in the time interval of the trajectory $\varphi_c(t,m_0)$ of $X_c,$ such that the segment of trajectory $\varphi_c([t_n,t_{n+1}],m_1)$ connects inside $S_n$ a point $m_n\in (A_nB_n)$ to a point $m_{n+1}\in (A_{n+1}B_{n+1})$.  As the vector field $X_c$ has a norm greater than a positive value if $x\geqslant v_0$, the time sequence $(t_i)$ tends to $-\infty$ (this means that $\tau_-(m_0)=-\infty)$. Now, as the vector field $X_c$ has a negative horizontal component in
the interior $\{x>0,y<0\}$  of the quadrant $Q,$ the orbit of $m_0$ has a regular projection onto the positive axis $Ox^+=\{x>0\}$ (the orbit is a covering of $Ox^+$). As $Ox^+$ is simply connected, this covering must be a diffeomorphism. The orbit is a graph of a function $y=y(x)$ defined for $x>0$. As $X_c$ is analytic, the trajectory is analytic and also the function $y(x)$ for $x>0$. Taking into account the properties of the local stable manifold, we already know that $y(x)$ has a continuous extension at $0$ by $y(0)=0$. Finally, we have by construction that $y_0(x)<y(x)<0$ for all $x>0$.
\end{proof}
\begin{figure}[htp]
\begin{center}
   \includegraphics[scale=0.6]{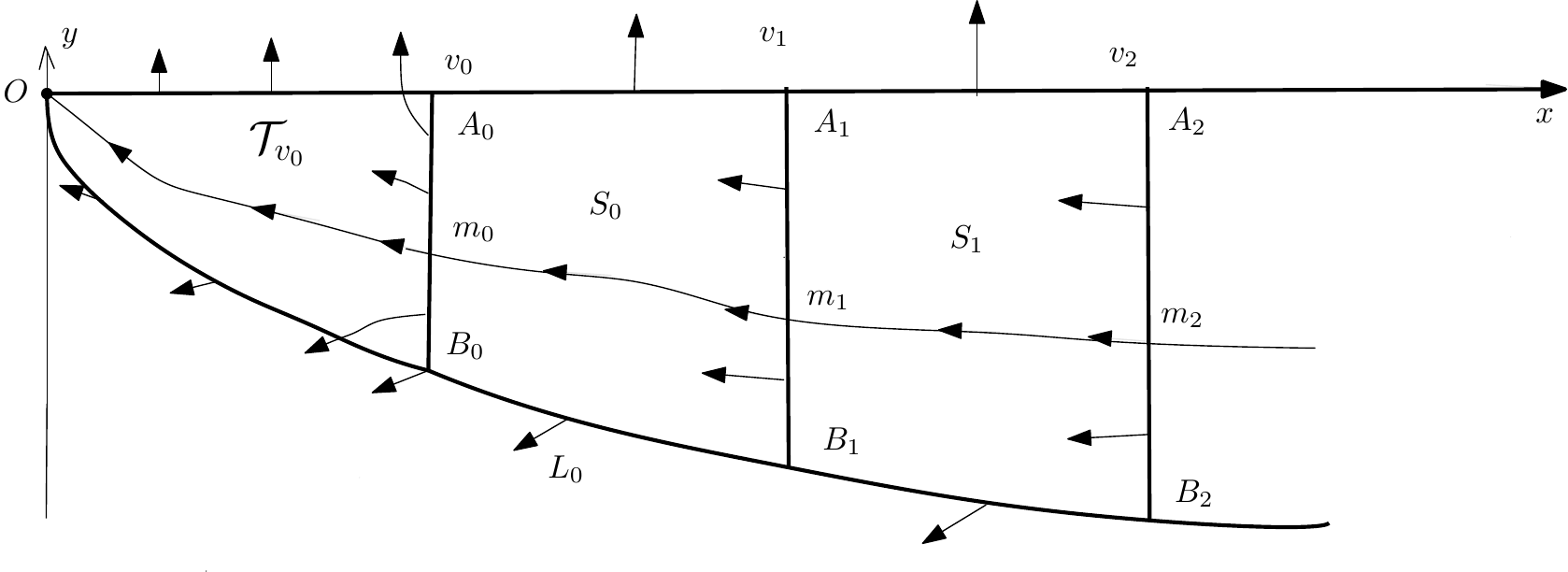}\\
  \caption{\it Construction of a global invariant manifold}\label{fig-glob-sta-man}
  \end{center}
\end{figure}
\subsection{Uniqueness of the stable manifold}\label{subsect-unicity}
Consider any stable manifold in $Q,$ i.e. any trajectory in $Q$ converging toward the origin.  As seen above, such a stable manifold exists, and we want to prove that only one exists.
\begin{lemma}\label{lem-unicity}
There exists a unique trajectory in $Q$ which converges toward the origin in positive times.
\end{lemma}
\begin {proof}
Let $\Gamma_1$ and $\Gamma_2$ be two such orbits.
With the same arguments as those used in Lemma \ref{lem-exist-global-s-m}, it is easily seen that there exists a $v_0>0$ such that these orbits are the graphs of two  functions $y_1(x)$ and $y_2(x)$ respectively, defined on the same interval $(0,v_0],$  verifying $y_1(x)<0$ and $y_2(x)<0$ for $0<x\leqslant v_0$ and  with a continuous extension $y_1(0)=y_2(0)=0$.  We want to prove that these two functions coincide on $[0,v_0]$, which implies that $\Gamma_1$ and $\Gamma_2$ also coincide globally.

Assume on the contrary, for instance, that $y_1(v_0)<y_2(v_0)$. As the two graphs are different half-orbits, this inequality persists for any $x\in (0,v_0]$, which means that
$y_1(x)<y_2(x)<0$ for  any $x\in (0,v_0]$. Moreover, each function is a solution of the differential equation $\mfrac{dy}{dx}= \mfrac{c}{\Lambda}+\mfrac{1}{\Lambda}\mfrac{x^\alpha}{y}$. Then, we have the following differential equation:
\begin{equation}\label{aep}
\frac{d}{dx}(y_2-y_1)=\frac{1}{\Lambda}\Big(\frac{1}{y_2}-\frac{1}{y_1}\Big)x^\alpha.
\end{equation}
As $y_1(x)<y_2(x)<0$,  we have that $\mfrac{1}{y_2(x)}-\mfrac{1}{y_1(x)}<0$ for all $0<x\leqslant v_0.$ It follows from \eqref{aep} that the function $x\mapsto (y_2-y_1)(x)$ is decreasing and, in particular:
\begin{equation*}
(y_2-y_1)(x)> (y_2-y_1)(v_0)>0
\end{equation*}
  for  all $x\in (0,v_0].$  This contradicts the fact that $(y_2-y_1)(x)\rightarrow 0$ when $x\rightarrow 0.$
\end{proof}
The unique stable manifold is the one obtained in Lemma \ref{lem-exist-global-s-m}.
 This manifold is the graph of a unique well-defined function we denote by $y_c(x)$. More specifically, we call it {\it the stable manifold of $X_c$ at the origin and write it in the form $\{y=y_c(x)\}$}. We recall that $y_c(x)$ is analytic for $x>0$, continuous at $0$ with value $y_c(0)=0$; it is located in $Q$ above the curve $L_0$, i.e.,  it verifies $y_0(x)<y_c(x)<0$ for all $x>0$. The new definition of $y_0(x)$  as the function $y_c(x)$ for  $c=0$, coincides with the former one.

\section{Some technical lemmata}\label{technical}
In Section \ref{sect-topological-approach}, we defined a $c$-family of analytic functions $y_c(x)$. Now, considering $c\geqslant0$ as a parameter, we study how this family depends on $c$ and prove some useful technical lemmata.

\subsection{Dependence of the stable manifold on parameter $c$}\label{subsect-c-dependence}

The vector field $X_c$ rotates in function of the parameter $c$, in the sense of Duff (see \cite{D}). We take advantage of this property in the following lemma:
\begin{lemma}\label{lem-yc-increasing}
For each $x>0$, the map $c\mapsto y_c(x) $ is increasing.
\end{lemma}
\begin{proof}
We consider two values of the parameter: $c_0<c_1.$  The stable manifolds are
$L_{c_0}:= \{y=y_{c_0}(x)\}$ and $L_{c_1}:= \{y=y_{c_1}(x)\}$, respectively. Take any value $x>0$; then, the vector
\begin{equation*}
 X_{c_1}(x,y_{c_1}(x))=\Big(y_{c_1}(x),\frac{c_1}{\Lambda}y_{c_1}(x)+\frac{1}{\Lambda}x^\alpha\Big)
 \end{equation*}
 is tangent to $L_{c_1}$  at the point $(x,y_{c_1}(x))$ for all $x>0.$  At the same point
 $(x,y_{c_1}(x)),$ the value of the vector field  $X_{c_0}$ is~:
\begin{equation*}
X_{c_0}(x,y_{c_1}(x))=\Big(y_{c_1}(x),\frac{c_0}{\Lambda}y_{c_1}(x)+\frac{1}{\Lambda}x^\alpha\Big).
\end{equation*}
As $\frac{c_0}{\Lambda}y_{c_1}(x)>\frac{c_1}{\Lambda}y_{c_1}(x),$ the vector field
$X_{c_0}$ is, for any $x>0$, transverse to the curve $L_{c_1}$ and is directed upward.
Then,  we construct trapping triangles for any $v_0$  (see Definition \ref{def-trap-triangle}), by using the curves $L^d=L_0$ and  $L^u=L_{c_1}.$ As a consequence of Lemma \ref{lem-general-trap-triangle},  the unique stable manifold for the vector field $X_{c_0}$, i.e. $L_{c_0}$, is located
strictly between $L_0$ and $L_{c_1}$ meaning that, for any $x>0,$ $y_0(x)<y_{c_0}(x)<y_{c_1}(x)$.
\end{proof}

Now look at the continuity of the stable manifold with respect to $c$.
\begin{lemma}\label{lem-yc-continuous}
For each $x>0$, the map $c \mapsto y_c(x) $ is continuous.
\end{lemma}
\begin{proof}
As  the map $c \mapsto y_c(x)$ increases, left- and right-hand limits exist at any $\bar c$, namely $y^-_{\bar c}(x)=\lim_{c\rightarrow {\bar c}_-}y_c(x)$ and $y^+_{\bar c}(x)=\lim_{c\rightarrow {\bar c}_+}y_c(x),$  with the property that $y^-_{\bar c}(x)\leqslant y_{\bar c}(x)\leqslant y^+_{\bar c}(x)$. To prove that the function $y_c$ is continuous at $\bar c$, we must check that $y^-_{\bar c}(x)=y^+_{\bar c}(x).$

We claim that the positive half-orbit of $X_{\bar c}$ by $y^-_{\bar c}(x)$ is a local stable manifold. If this is not the case, it follows from Lemma \ref{lem-exist-local-s-m} that the trajectory must cut (transversally)  $L_0^+=L_0\setminus\{O\}$ or $Ox^+=Ox\setminus\{O\}$ after a finite time $t_0$.
Choosing some $t_1>t_0$, we apply Theorem \ref{th-Cauchy} at  the segment  of trajectory $\varphi([0,t_1],(x,y^-_{\bar c}),\bar c)$.   Theorem \ref{th-Cauchy} implies that the map
$c \mapsto \{t\in [0,t_1]\mapsto \varphi(t,(x,y_c(x)),c)\}$
from  $(\bar c-\delta,\bar c]$ to ${C}^1([0,t_1],\R^2)$ exists and  is continuous at the endpoint $\bar c$ if $\delta$ small enough. Then, the orbit of $X_c$ by the point $(x,y_c(x))$ for $c\in (\bar c-\delta,\bar c)$ and $\delta$ small enough, cuts also  $L_0^+$ or $Ox^+.$ However, this is impossible because $\{y=y_c(x)\}$ is the stable manifold of $X_c.$ Then, the only possibility for the trajectory through the point $(x,y^-_{\bar c}(x))$ is the invariant manifold of $X_{\bar c}.$
A similar argument proves that the trajectory through the point $(x,y^+_{\bar c}(x))$ has also to be the invariant manifold of $X_{\bar c}.$ As the invariant manifold of $X_{\bar c}$ is unique, we obtain $y^-_{\bar c}(x)=y^+_{\bar c}(x).$
\end{proof}
Finally, the family of maps $y_c(x)$ enjoys the following global properties:
\begin{lemma}\label{lem-global-yc}
The mapping $(x,c)\in [0,+\infty)\times  [0,+\infty)\mapsto y_c(x)\in\R$ is continuous as a function of two variables and for all $c$, one has $y_c(0)=0.$ For all $c$, the function $x\mapsto y_c(x)$ is analytic for $x>0$.  Recall that, for all $x>0$, the map $c\mapsto y_c(x)$ increases continuously.
\end{lemma}
\begin{proof}
Consider a fixed value $x_0>0$ and write $y_0=y_{c_0}(x_0)$. The trajectory
$X_{c_0}$ by the point $(x_0,y_0)$ cuts transversally the vertical line $d_x$ passing by any $(x,0)$ with $x>0,$ at the point $(x,y_{c_0}(x)).$ Then, if $(x_1,y_1,c)$ is near $(x_0,y_0,c_0)$, the trajectory of
$X_c$ by the point $(x_1,y_1)$  cuts also $d_x$ transversally. This defines a local function $y=F(x_1,y_1,x,c)$ in a neighborhood of the point $(x_0,y_0,x,c)\in \R^4.$ Using Theorem \ref{th-Cauchy} and the Implicit Function Theorem we have $F$ is analytic.
Now, we observe that $y_c(x)=F(x_0,y_c(x_0),x,c)$,  where  $x_0$ is fixed. By Lemma \ref{lem-yc-continuous}, the function $c\mapsto y_c(x_0)$ is continuous. Then,  the composed function $(x,c)\mapsto y_c(x)$ is also continuous.  This proves the continuity at $(x,c)$ for any $x>0$.  As $y_0(x)\leqslant y_c(x)<0$, the function $y_c(x)$ converges uniformly in $c$ toward $0$ when $x\rightarrow 0_+.$

The  function $x\mapsto y_c(x)=F(x_0,y_c(x_0),x,c)$ is analytic as a partial function
of $F.$
\end{proof}
\subsection{An asymptotic expression for $y_c(x)$}\label{asymptotic}
\begin{lemma}
For all $c>0$ and $x>0$, the following holds:
\begin{equation}\label{eq-lc-trapping}
-\Big(\frac{2}{(1+\alpha)\Lambda}\Big)^{1/2}x^\frac{1+\alpha}{2}<y_c(x)<-\Big(\frac{2}{(1+\alpha)\Lambda}\Big)^{1/2}x^\frac{1+\alpha}{2}+\frac{c}{\Lambda}x.
\end{equation}
Then, $y_c$ has the following asymptotic expansion at the origin:
\begin{equation}\label{eq-yc-asympt}
 y_c(x)=-\Big(\frac{2}{(1+\alpha)\Lambda}\Big)^{1/2}x^\frac{1+\alpha}{2}+O(x),
 \end{equation}
 where   the term $O(x)$ is uniform in $c$ when $c$ belongs to some compact subset  of $[0,+\infty)$.
\end{lemma}
\begin{proof}We claim that along the curve
\begin{equation}\label{cl}
l_c:= \{y=\tilde y_c(x)=-\Big(\frac{2}{(1+\alpha)\Lambda}\Big)^{1/2}x^\frac{1+\alpha}{2}+\frac{c}{\Lambda}x\},
\end{equation}
the vector field $X_c$ is transverse, with an upward direction. To prove this claim, we consider the field $N_1(x)$ obtained by rotating the tangent field to the graph of $l_c$ at the point  $(x,\tilde y_c(x))$ by the angle $\pi/2$:
$$N_1(x)=\Big(\Big(\frac{1+\alpha}{2\Lambda}\Big)x^\frac{\alpha-1}{2}-\frac{c}{\Lambda}, 1\Big),$$
and we compute its scalar product   with the vector $X_c(x)=X_c(x,\tilde y_c(x))$ which is equal to:
$$X_1(x)=\Big(-\Big(\frac{2}{(1+\alpha)\Lambda}\Big)^{1/2}x^\frac{1+\alpha}{2}+\frac{c}{\Lambda}x,-\frac{c}{\Lambda}\Big(\frac{2}{(1+\alpha)\Lambda}\Big)^{1/2}x^\frac{1+\alpha}{2}+\frac{c^2}{\Lambda^2}x+\frac{1}{\Lambda}x^\alpha\Big).$$
We obtain
$$<X_1(x),N_1(x)> =\frac{c}{\Lambda}\Big(\frac{1+\alpha}{2\Lambda}\Big)^{1/2}x^\frac{1+\alpha}{2},$$
which implies the upward transversality all along $l_c\setminus\{0\}$.

Notice that the curve $l_c$ cuts the $Ox$-axis at the value  $x_1(c)=\Big(\mfrac{2\Lambda}{(1+\alpha)c^2}\Big)^\frac{1}{1-\alpha}>0$ and remains  in the quadrant $Q$ only for $x\in [0,x_1(c)]$. Nevertheless, we can construct trapping triangles, using the curves $L_0$ and $l_c,$ with a vertical side in $\{x=v_0\}$ when $0<v_0\leqslant x_1(c)$.
It follows from Lemma \ref{lem-general-trap-triangle} that $y_c(x)$ verifies (\ref{eq-lc-trapping}) for $0<x\leqslant x_1(x).$
As the graph of $l_c$ is above the $Ox$ axis, (\ref{eq-lc-trapping}) is trivially verified if $x\geqslant x_1(c).$

The asymptotic formula (\ref{eq-yc-asympt}) follows directly from (\ref{eq-lc-trapping}).
\end{proof}

\subsection{Estimate from below}
The following estimate is significant in view of Lemma \ref{convex}.
\begin{lemma}\label{from_below}
The stable manifold  is located above the curve $\{y=-\frac{1}{c}x^{\alpha}\}.$ This means that for all $x>0$ it holds $y_c(x)>-\frac{1}{c}x^{\alpha}$.
\end{lemma}
\begin{proof}
Along the curve  $\{y=-\frac{1}{c}x^{\alpha}\}$, the vector field $X_c$ is equal to $y\frac{\partial}{\partial x}$: it is horizontal and directed toward the left for all $x>0.$  As the tangent vector field to the curve  $\{y=-\frac{1}{c}x^{\alpha}\}$
has a non-zero vertical component equal to $-\alpha x^{\alpha-1}$ for all $x>0$, it follows that $X_c$ is transverse to the curve $\{y=-\frac{1}{c}x^{\alpha}\}$ and directed downward all along it.  Then, for any $v_0>0$,  we can construct  a trapping triangle $\mathcal{T}_{v_0}$, by using the curves  $L^d=\{y=-\frac{1}{c}x^{\alpha}\}$ and  the $Ox$-axis as $L^u$.  By Lemma \ref{lem-general-trap-triangle}, it follows that the stable trajectory is trapped inside $\mathcal{T}_{v_0}$ for any $v_0>0.$ This concludes the proof.
\end{proof}

\section{Proof of Theorem \ref{theorem_scalar}}\label{proof_thm_scalar}
We are now in the position of proving Theorem \ref{theorem_scalar}. Let $0<\alpha<1$, $\Lambda>0$ be fixed, $0<v_0<1$ according to the physical framework.

\subsection{Existence of $c(v_0)$}\label{initial-th_scalar}
Recall that the initial condition \eqref{IC} reads
\begin{equation}\label{IC_bis}
x(0)=v_0, \quad y(0) =\displaystyle -\frac{c}{\Lambda}(1-v_{0}).
\end{equation}
With the notation of Section \ref{sect-topological-approach}, fulfilling the initial condition \eqref{IC_bis} is equivalent to the existence of a $c=c(v_0)$ such that
\begin{equation}\label{IC_ter}
y_{c(v_0)}(v_0)=-\frac{c(v_0)}{\Lambda}(1-v_0).
\end{equation}
The difficulty seems to be that $c(v_0)$
appears in both sides of the equation \eqref{IC_ter}, i.e., in the expression of the stable manifold and in the initial condition itself. However, taking into account the above properties of $y_c(x)$, this is not really an issue, as we will see below.
\begin{lemma}\label{lem-init-cond}
For any $v_0$, $0<v_0<1,$ there exists a unique value $c(v_0)$ such that
$$y_{c(v_0)}(v_0)=-\frac{c(v_0)}{\Lambda}(1-v_0).$$
Moreover, the functions $v_0\mapsto c(v_0)$ and $v_0 \mapsto y_{c(v_0)}(x)$ are continuous on $(0,1)$.
\end{lemma}
\begin{proof}
For simplicity, we denote $y(c)=y_c(v_0)$ and $\bar y(c)= -\frac{c}{\Lambda}(1-v_0)$ for a fixed $v_0$.
 At $c=0$ it holds  $y(0)=y_0(v_0)=-\Big(\mfrac{2}{(1+\alpha)\Lambda}\Big)^{1/2}v_0^\frac{1+\alpha}{2}<\bar y(0)=0$. When $c$ increases, $y(c)$ increases and $\bar y(c)$ decreases. When $\bar y(c)$ reaches the value $y_0(v_0)$,  $y(c)$ is larger than $y_0(v_0)$.

Consider the function
 \begin{equation*}
 \psi(v_0,c)=y_c(v_0)+\frac{c}{\Lambda}(1-v_0).
 \end{equation*}
To prove the existence and uniqueness of $c(v_0),$ we just need to apply the Intermediate Value Theorem  to the increasing map $c \mapsto\psi(v_0,c)$, which rises from a negative value for $c=0$ to a positive one when $c$ is large enough.

 We also use  the function $\psi$ to prove the continuity of the map $v_0\mapsto c(v_0).$ To this end,
 we take any $v_0^0\in (0,1)$
 and any $\varepsilon>0$ such that $[c(v_0^0)-\frac{\varepsilon}{2},c(v_0^0)+\frac{\varepsilon}{2}]\subset \R^+$. We have that
 \begin{equation*}
 \psi\Big(v_0^0,c(v_0^0)-\frac{\varepsilon}{2}\Big)<\psi\Big(v_0^0,c(v_0^0)\Big)=0<
 \psi\Big(v_0^0,c(v_0^0)+\frac{\varepsilon}{2}\Big).
 \end{equation*}
 As the map $v_0\mapsto \psi(v_0,c)$ is continuous by Lemma \ref{lem-global-yc}, a $\delta(\varepsilon)>0$ exists, such that, if $|v_0-v_0^0| <\delta(\varepsilon)$, we have that $\psi\Big(v_0,c(v_0^0)-\frac{\varepsilon}{2}\Big)<0$ and $\psi\Big(v_0,c(v_0^0)-\frac{\varepsilon}{2}\Big)>0.$ Then, as a consequence of the  Intermediate Value Theorem, the unique solution $c(v_0)$ of the equation $\psi(v_0,c)=0$ is contained in the open interval  $(c(v_0^0)-\frac{\varepsilon}{2},c(v_0^0)+\frac{\varepsilon}{2}).$ As $\varepsilon$ can be taken arbitrarily small, this completes the proof of the continuity of $c(v_0)$ at the point $v_0^0.$
 The continuity of the map $v_0\mapsto y_{c(v_0)}(v_0)$ follows from the continuity of the map $c(v_0)$ and  the continuity of the map  $(c,v_0)\mapsto y_c(v_0)$ proved in Lemma \ref{lem-global-yc}.
 \end{proof}
\subsection{Estimate of $c(v_0)$}\label{estimate_c(v_0)}
To estimate $c(v_0),$ it is interesting to introduce
 the control line $l_c$ defined by \eqref{cl} which we have already used in the previous subsection. We now consider the motions of the three points $y(c)=y_c(v_0)$, $\tilde y(c)= \tilde y_c(v_0)$ and $\bar y(c)$. We have that
 $y(c)<\tilde y(c)$. Recall that $\bar y(c)$ is decreasing and  $y(c)$, $\tilde y(c)$ are increasing, with $y(c)<\tilde y(c)$ for $c>0$ .

 At $c=0$ we have the inequality $y(0)=\tilde y(0)< \bar y(0)=0$.  Subsequently, $\bar y(c)$ crosses $\tilde y(c)$ at a value $c=c_-(v_0)$ at which it holds $y(c_-(v_0))< \tilde y(c_-(v_0))=\bar y(c_-(v_0))$. When $c$ eventually reaches the value $c_+$ where $\bar y(c_+)=y_0(v_0),$ we have the inequality $\bar y(c_+)<y(c_+)<\tilde y(c_+)$.
 This yields the following bounds for $c(v_0)$:
 \begin{equation}\label{bounds}
 c_-(v_0)<c(v_0)<c_+(v_0),
 \end{equation}
 where $c_+=c_+(v_0)$ is given by
  \begin{equation}
 -\frac {c_+}{\Lambda} (1-v_0)=-\Big(\frac{2}{(1+\alpha)\Lambda}\Big)^{1/2}v_0^\frac{1+\alpha}{2},
  \end{equation}
 and $c_-=c_-(v_0)$ by

  \begin{equation}
 -\frac {c_-}{\Lambda} (1-v_0)=-\Big(\frac{2}{(1+\alpha)\Lambda}\Big)^{1/2}v_0^\frac{1+\alpha}{2}+\frac{c_-}{\Lambda}v_0.
  \end{equation}
 We  find that $c_+(v_0)=\Big(\frac{2\Lambda}{1+\alpha}\Big)^{1/2}\frac{v_0^\frac{1+\alpha}{2}}{1-v_0}$ and  $c_-(v_0)=\Big(\frac{2\Lambda}{1+\alpha}\Big)^{1/2}v_0^\frac{1+\alpha}{2}$. These values coincide with the a priori estimates (\ref{upper-c2}) and  (\ref{lower-c2}) obtained by a different method in Subsection \ref{formal_estimate}.

Finally, we observe that Lemma \ref{from_below} provides us with the following new estimate:
\begin{equation}\label{new_estimate}
 \displaystyle \frac{c}{\Lambda}(v_0-1)>-\frac{1}{c}v_0^{\alpha} \quad\; \textrm{or equivalently} \quad\; c(v_0)< \Big(\frac{\Lambda}{1-v_0}\Big)^{1/2} v_0^{\frac{\alpha}{2}}.
 \end{equation}
 It is readily seen that
 $\Big(\mfrac{\Lambda}{1-v_0}\Big)^{1/2} v_0^{\frac{\alpha}{2}}>c_-(v_0)=\Big(\mfrac{2\Lambda}{1+\alpha}\Big)^{1/2}v_0^\frac{1+\alpha}{2}$,
 because it holds
 $\displaystyle v_0(1-v_0)<(1+\alpha)/2$ as $0<v_0<1$ and $0<\alpha<1$.
 Therefore, combining \eqref{new_estimate} with the bounds \eqref{bounds} reveals that
 \begin{equation}
c_-(v_0)<c(v_0)<\min \left(c_+(v_0),\Big(\frac{\Lambda}{1-v_0}\Big)^{1/2} v_0^{\frac{\alpha}{2}} \right).
 \end{equation}

 \subsection{Estimations for the time $R(v_0)$} \label{estimations_time}
Taking into account the existence and uniqueness of the stable manifold $L_c := \{y=y_c\}$ obtained in Subsection \ref{subsect-c-dependence} (see Lemma \ref{lem-global-yc}), it remains to prove that the settling time $R(v_0)$, along the stable manifold, to proceed from the initial condition $(v_0, y_{c(v_0)}(v_0))$ to the origin is finite.

We first consider the settling time $T(x,c)$ to go from $(x,y_c(x))$ to the origin, along the stable manifold $L_c$ of $X_c.$
 \begin{lemma}\label{lem-time-T}
 The settling time  $T(x,c)$  is finite for any $(x,c)\in \R^2_+=  [0,+\infty)\times [0,+\infty)$. Moreover, $T(x,c)$ is continuous on $\R^2_+$ and such that $T(0,c)= 0$. In the interval
  $0<x<x_1(c)=\Big(\mfrac{2\Lambda}{(1+\alpha)c^2}\Big)^\frac{1}{1-\alpha}$, we have the following bounds for $T(x,c)~$

  \begin{equation}\label{eq-T-framing}
 \frac{(2(1+\alpha)\Lambda)^{1/2}}{1-\alpha}x^\frac{1-\alpha}{2}<T(x,c)< \frac{(2(1+\alpha)\Lambda)^{1/2}}{1-\alpha}x^\frac{1-\alpha}{2}\Big(1-c\Big(\frac{1+\alpha}{2\Lambda}\Big)^{1/2}x^\frac{1-\alpha}{2}\Big)^{-1},
 \end{equation}
 which gives the following asymptotic expansion at $(0,c)$:
 \begin{equation}\label{eq-T-asympt}
 T(x,c)= \frac{(2(1+\alpha)\Lambda)^{1/2}}{1-\alpha}x^\frac{1-\alpha}{2}+O(cx^{1-\alpha}),
 \end{equation}
 where the remainder is uniform in $c$ when $c$ is restricted to a compact subset of $[0,+\infty)$.
 \end{lemma}
 \begin{proof}
 The settling time $T(x,c)$ is obtained by integrating the equation $dt=\mfrac{dx}{y_c(x)}$ from $x$ to $0$. Changing the sense of integration, we arrive that $T(x,c)=\int_0^x\mfrac{ds}{|y_c(s)|}.$ This integral is improper at $x=0$, but as by (\ref{eq-yc-asympt}) the function $|y_c(x)|$ is equivalent to $\Big(\mfrac{2}{(1+\alpha)\Lambda}\Big)^{1/2}x^\frac{1+\alpha}{2},$ with $\mfrac{1+\alpha}{2}<1,$ the integral converges and is continuous in $x$. To have continuity in $(x,c)$, we require the above equivalence to be uniform in $c$ when $c$ belongs to a compact subset of $[0,+\infty)$.

 We have that
  \begin{equation}\label{eq-framing}
 \Big(\frac{2}{(1+\alpha)\Lambda}\Big)^{1/2}s^\frac{1+\alpha}{2}
 -\frac{c}{\Lambda}s<|y_c(s)|<
 \Big(\frac{2}{(1+\alpha)\Lambda}\Big)^{1/2}s^\frac{1+\alpha}{2}.
  \end{equation}
Moreover, if  $s\leqslant x<x_1(c)=\Big(\mfrac{2\Lambda}{(1+\alpha)c^2}\Big)^\frac{1}{1-\alpha},$ we have:
$$ \Big(\frac{2}{(1+\alpha)\Lambda}\Big)^{1/2}s^\frac{1+\alpha}{2}
 -\frac{c}{\Lambda}s>\Big(\frac{2}{(1+\alpha)\Lambda}\Big)^{1/2}s^\frac{1+\alpha}{2}\Big(1-c\Big(\frac{1+\alpha}{2\Lambda}\Big)^{1/2}\Big)x^\frac{1-\alpha}{2}\Big).
 $$
 As $\int_0^x s^{-\frac{1+\alpha}{2}}ds
 =\mfrac{2}{1-\alpha}x^\frac{1-\alpha}{2},$ the last inequality and (\ref{eq-framing}) implies the estimates (\ref{eq-T-framing}).
The asymptotic expansion follows directly from (\ref{eq-T-framing}).
\end{proof}

We now consider  the time $R(v_0)=T(v_0,c(v_0))$ corresponding to the initial conditions \eqref{IC_bis}. A direct consequence of Lemma \ref{lem-time-T} and the continuity of the mapping $v_0\mapsto c(v_0)$  is the following:
\begin{lemma}\label{lem-time-R}
The time  $R(v_0)$ to proceed from $\Big(v_0,-\mfrac{c}{\Lambda}(1-v_{0})\Big)$ to the origin,  along the stable manifold  of $X_{c(v_0)},$ is finite for any $v_0\in [0,1).$  Moreover, $R(v_0)$ is continuous on $[0,1)$  and $R(0)= 0.$
\end{lemma}
\begin{proof} As $c(v_0)<c_+(v_0)=\Big(\frac{2\Lambda}{1+\alpha}\Big)^{1/2}\frac{v_0^\frac{1+\alpha}{2}}{1-v_0},$ we have that $c(v_0)=O\Big(v_0^\frac{1+\alpha}{2}\Big).$ Then, $c(v_0)$ is continuous on $[0,1)$ with $c(0)=0.$  It follows that $R(v_0)$ is continuous on $[0,1)$ with the value $R(0)=0.$
\end{proof}
As the curve $L_c$ is above the curve $L_0,$ the inequality $ \frac{(2(1+\alpha)\Lambda)^{1/2}}{1-\alpha}x^\frac{1-\alpha}{2}<T(x,c)$ in (\ref{eq-T-framing}) is in fact true for all $x>0$. Then we have the following lower bound for $R(v_0):$
\begin{equation}\label{eq-lower-bound}
R(v_0)>\frac{(2(1+\alpha)\Lambda)^{1/2}}{1-\alpha}v_0^\frac{1-\alpha}{2}
\end{equation}
for $0<v_0<1$ and $0\leqslant \alpha<1$.

To obtain an upper bound for $R(v_0)$ which would be valid for any $0<v_0<1,$ it would be better to replace the curve $l_{c(v_0)}$ by one which remains below $\{y=0\}$.
The simplest solution is a curve $d_k:=\{ y=\bar y_k=-k x^\frac{1+\alpha}{2}\}$, with $k=k(v_0)$ chosen so that the vector field $X_{c(v_0)}$ is upward transverse to $d_k$, at least for $x\in (0,v_0].$ A first condition on $k$ is clearly $d_k$ above $L_0$, i.e.
\begin{equation}\label{eq-cond1-k}
k^2< \frac{2}{(1+\alpha)\Lambda}.
\end{equation}
To express the transversality condition, we write the vector field  $X_c$ along $d_k~:$

$$X(x)=(-kx^\frac{1+\alpha}{2}, -k\frac{c}{\Lambda}x^\frac{1+\alpha}{2}+\frac{1}{\Lambda}x^\alpha),$$
and the orthogonal vector field to $d_k$ in the upward direction, given by~:
$$N(x)=(\frac{1+\alpha}{2}kx^\frac{\alpha-1}{2},1).$$
Then,  we compute their scalar product~:
\begin{equation*}
<X(x),N(x)>=\Big(\frac{1}{\Lambda}-\frac{1+\alpha}{2}k^2\Big)x^\alpha-k\frac{c}{\Lambda}x^\frac{1+\alpha}{2}.
\end{equation*}
Upon factoring, it comes:
\begin{equation*}
<X(x),N(x)>=\Big(\frac{1}{\Lambda}-\frac{1+\alpha}{2}k^2\Big)x^\alpha\Big[1-k\frac{c}{\Lambda}\Big(\frac{1}{\Lambda}-\frac{1+\alpha}{2}k^2\Big)^{-1}x^\frac{1-\alpha}{2}\Big].
\end{equation*}
The condition (\ref{eq-cond1-k}) implies that the factor $\Big(\frac{1}{\Lambda}-\frac{1+\alpha}{2}k^2\Big)$ is positive. Then, for the scalar product to be positive for $c=c(v_0)<c_+(v_0)$ and $x\leqslant v_0$,
it is sufficient that the following condition be fulfilled:
\begin{equation*}
k\frac{c_+(v_0)}{\Lambda}\Big(\frac{1}{\Lambda}-\frac{1+\alpha}{2}k^2\Big)^{-1}v_0^\frac{1-\alpha}{2}<1.
\end{equation*}
Taking into account the expression of $c_+(v_0)$ this condition reads as follows:
\begin{equation}\label{eq-cond2_k}
k\Big(\frac{2\Lambda}{1+\alpha}\Big)^{1/2}<\frac{1-v_0}{v_0}\Big(1-\frac{(1+\alpha)\Lambda}{2}k^2\Big).
\end{equation}
From now on, we also assume  that
\begin{equation}\label{eq-cond2-k}
k\leqslant \Big(\frac{1}{(1+\alpha)\Lambda}\Big)^{1/2},
\end{equation}
which implies the condition (\ref{eq-cond1-k}).
As $\frac{1}{v_0}>1,$ to have (\ref{eq-cond2_k}), assuming that condition (\ref{eq-cond2-k}) is fulfilled,  it suffices to have the following condition:
\begin{equation}\label{eq-cond3-k}
k\leqslant \frac{1}{2}\Big(\frac{1+\alpha}{2\Lambda}\Big)^{1/2}(1-v_0).
\end{equation}
As
\begin{equation}\label{eq-A}
A(\alpha)=\frac{2^{3/2}}{(1+\alpha)^{1/2}}> (1+\alpha)^{1/2},
\end{equation}
it is easy to see that conditions (\ref{eq-cond2-k}) and (\ref{eq-cond3-k}) are verified if
\begin{equation}\label{eq-value-k}
 k=\frac{1-v_0}{A(\alpha)\Lambda^{1/2}}.
\end{equation}
For this value of $k,$ the vector field $X_{c(v_0)}$ is transverse in the upward direction along the curve $d_k$ for all $x\in (0,v_0]$. Accordingly, we construct  a trapping triangle, which implies that, for all $x\in (0,v_0]$, it holds $y_{c(v_0)}< \bar y_k(x)=-k x^\frac{1+\alpha}{2}$.  This  implies that $R(v_0)$ has an upper bound equal  to $\frac{1}{k} \frac{1}{1-\alpha}v_0^\frac{1-\alpha}{2}$ for all $v_0\in [0,1)$.  Thanks to (\ref{eq-lower-bound}) and (\ref{eq-value-k}), we have the following bounds:
\begin{lemma}\label{lem-bounds-v0}
 For $0<v_0<1$ and $0\leqslant \alpha <1$, we have that
\begin{equation}\label{eq-bounds-v0}
\frac{(2(1+\alpha)\Lambda)^{1/2}}{1-\alpha}v_0^\frac{1-\alpha}{2}<R(v_0)< \frac{2\Lambda^{1/2}A(\alpha)}{1-\alpha}\frac{v_0^\frac{1-\alpha}{2}}{1-v_0},
\end{equation}
where $A(\alpha)$ is defined by (\ref{eq-A}).
\end{lemma}

\subsection{End of the proof of Theorem \ref{theorem_scalar} and Corollary \ref{corollary_thm_scalar}}
To complete the proof of Theorem \ref{theorem_scalar}, the two following points remain to be established. We return to the original notation of the free boundary problem \eqref {scalar-v}, see \eqref{notation_syst_dyn}.
\begin{lemma}\label{holder}
For $0<\alpha<1$, it holds $v\in C^{\infty}([0,R))\cap C^{2+[\beta],\beta-[\beta]}([0,R])$, $\beta=\mfrac{2\alpha}{1-\alpha}$.
\end{lemma}
\begin{proof}
It is clear that $v$ is smooth as long as it does not vanish. We need to establish the H\"older regularity of the function $v(\xi)$ near $\xi=R$. It follows from Subsection \ref{asymptotic} that near the origin, it holds
\begin{equation}\label{equiv}
y_c(x)\sim -\Big(\frac{2}{(1+\alpha)\Lambda}\Big)^{1/2}x^\frac{1+\alpha}{2},
\end{equation}
that is, near $\xi=R_-$,
\begin{equation}
v'(\xi) \sim -\Big(\frac{2}{(1+\alpha)\Lambda}\Big)^{1/2}v^\frac{1+\alpha}{2}(\xi).
\end{equation}
An elementary integration yields
\begin{equation}
v(\xi) \sim \frac{1-\alpha}{2}\Big(\frac{2}{(1+\alpha)\Lambda}\Big)^{1/2}\Big(R-\xi\Big)^{\mfrac{2}{1-\alpha}}, \quad {\mfrac{2}{1-\alpha}}=2+\beta.
\end{equation}
\end{proof}

\begin{lemma}\label{convex}
The function $v(\xi)$ is decreasing, convex on the interval $[0,R)$.
\end{lemma}
\begin{proof}
On one hand, it follows from Lemma \ref{from_below} that the stable manifold $y_c$ never cuts the curve $y= -\mfrac{1}{c}x^{\alpha}$, that is, $y'_c$ does not vanish. In other words, $v(\xi)$ has no inflection points. On the other hand, Lemma \ref{holder} yields that $v(\xi)$ is convex near $\xi=R$, $\xi<R$, hence the convexity of $v(\xi)$ on $(0,R)$.
\end{proof}	

Finally, we prove Corollary \ref{corollary_thm_scalar}. Let $0<a<b<1$, $I=[a,b]$. According to Lemma \ref {lem-init-cond} and Lemma \ref{lem-time-R}, respectively, the mappings $v_0 \mapsto c(v_0)$ and $v_0 \mapsto R(v_0)$ are continuous from $I \rightarrow \R$. Thanks to Lemma \ref{lem-bounds-v0}, we have the uniform estimate:
\begin{equation}\label{uniform_R}
\frac{(2(1+\alpha)\Lambda)^{1/2}}{1-\alpha}a^{\frac{1-\alpha}{2}}<R(v_0)\leqslant R_{max}= \frac{2\Lambda^{1/2}A(\alpha)}{1-\alpha}\frac{b^\frac{1-\alpha}{2}}{1-a}.
\end{equation}
Let $\widetilde{v}(v_0)$ be the extension by $0$ of the function $\xi \mapsto v(v_0;\xi)$ to the interval $[0,R_{max}]$, clearly $\widetilde{v}(v_0) \in C^1([0,R_{max}])$.

\begin{lemma}\label{continuity}
The mapping $v_0 \mapsto \widetilde{v}(v_0)$ is continuous from $I$ to $C^1([0,R_{max}])$.
\end{lemma}
\begin{proof} 
The mapping $\tilde v(v_0)$ can be seen as the function of two variables
 $\tilde v(v_0;\xi)$. The mapping  $v_0 \mapsto \widetilde{v}(v_0)$ is continuous from $I \to C^1([0,R_{max}])$ if and only if   $\tilde v(v_0;\xi)$ is continuous on the rectangle $T=[0,R_{max}]\times I$, which  is cut by the curve $l_R:=v_0\rightarrow R(v_0)$ into two closed subsets $T_1=\{v_0\in I, 0\leqslant \xi\leqslant R(v_0)\}$ and $T_2=\{v_0\in I, R(v_0)\leqslant \xi\leqslant R_{max}\}$. The function $\tilde v(v_0;\xi)$ is continuous on $T_2$ where it is identically $0$ and on $T_1$ as it is equal to the $x$-component $x(\xi,(v_0,y_{c(v_0)}(v_0)))$ of the trajectory  of $X_{c(v_0)}$ with initial conditions $(v_0,y_{c(v_0)}(v_0))$. These two definitions have the same value $0$ along $l_R=T_1\cup T_2$. Then, they  define a continuous function on $T$.
\end{proof}


\section{The free boundary problem as an integro-differential system}\label{well-posedness}
 Here, we revisit the one-phase free boundary problem \eqref{free_boundary_problem}-\eqref{free_boundary_conditions} that we rewrite in the form of two coupled subsystems:
\begin{equation}\label{scalar-v-bis}
\begin{cases}
\Lambda v''-cv'=v^{\alpha},\quad 0<\xi<R,\\[1mm]
0<v(0)<1, \, v'(0)= -\displaystyle\frac{c}{\Lambda}(1-v(0)),\\[1mm]
v(R) = v'(R)=0,
\end{cases}
\end{equation}
and
\begin{equation}\label{scalar-u}
\begin{cases}
u''-cu'=-v^{\alpha},\quad 0<\xi<R,\\[1mm]
u(0)=\theta,\, u'(0)=c\theta,\\[1mm]
u(R)=1,\quad u'(R)=0.
\end{cases}
\end{equation}
\begin{remark} A first sight, subsystems \eqref{scalar-v-bis} and \eqref{scalar-u} seem uncoupled because $v$ appears in the right hand side of \eqref{scalar-u}. However, the linkage is through $c$ (see the explicit computation in Subsection \ref{Lambda=0} when $\Lambda=0$).
\end{remark}

\subsection{Equivalence with an integro-differential system}\label{equivalence}
We have the following proposition.
\begin{proposition}\label{integral}
Assume there is a solution $(c,R,u(\xi),v(\xi))$ to system \eqref{scalar-v-bis}-\eqref{scalar-u} such that $c>0$, $0<R<+\infty$, $u$ and $v$ in $C^{\infty}([0,R))\cap C^1([0,R])$. Then,
$(c,R,v(\xi))$ verifies the system consisting of \eqref{scalar-v-bis} and one of the following integral equations:
\begin{align}\label{equiv-1}
\theta+\Lambda v(0)&=1-c(1-\Lambda)\int_0^Re^{-cs}v(s)\, ds,\\
\label{equiv-2}
\theta+v(0)&=1-(1-\Lambda)\int_0^Re^{-cs}v'(s)\, ds.
\end{align}
Furthermore, the three systems \eqref{scalar-v-bis}-\eqref{scalar-u}, \eqref{scalar-v-bis}\&\eqref{equiv-1} and \eqref{scalar-v-bis}\&\eqref{equiv-2} are equivalent.
\end{proposition}
\begin{proof}
	(i) Let $(c,R,u(\xi),v(\xi))$ be a solution to \eqref{scalar-v-bis}-\eqref{scalar-u} as above.
	Set $w=u+v$, one infers from system \eqref{scalar-v-bis}-\eqref{scalar-u} that $w$ satisfies
	\begin{equation}\label{scalar-w}
	\begin{cases}
	w''-cw'=(1-\Lambda)v'',\quad 0<\xi<R,\\
	w(R)=1,\, w'(R)=0.
	\end{cases}
	\end{equation}
	For any $\xi\in (0,R)$, we integrate \eqref{scalar-w} from $\xi$ to $R$. Using the free boundary conditions in \eqref{scalar-v-bis} and in \eqref{scalar-w}, we arrive at
	\begin{equation}\label{scalar-w1}
	w'-cw=(1-\Lambda)v'-c, \quad w(R)=1.
	\end{equation}
	It is easy to solve \eqref{scalar-w1} and obtain:
	\begin{equation}\label{soluw}
	w(\xi)=e^{-c(R-\xi)}-\int_{\xi}^Re^{-c(s-\xi)}\big((1-\Lambda)v'(s)-c \big) ds.
	\end{equation}
	Taking $\xi=0$ in \eqref{soluw}, it is not difficult to recover \eqref{equiv-2} using the boundary conditions in \eqref{scalar-v-bis} and \eqref{scalar-u}. Likewise, setting $\widetilde{w}=u+\Lambda v$, one proceeds in the same manner to obtain \eqref{equiv-1} (or equivalently we may integrate by part the integral in \eqref{equiv-2}).
	
	(ii) Assume that $(c,R,v(\xi))$ verifies \eqref{scalar-v-bis}, $v \in C^{\infty}([0,R))\cap C^1([0,R])$, and in addition \eqref{equiv-2} holds. We construct a unique solution to \eqref{scalar-w} as in \eqref{soluw}. Next, setting $u=w-v$, $u$ is in $C^{\infty}([0,R))\cap C^1([0,R])$ and satisfies the equation in \eqref{scalar-u}. From the free boundary conditions in \eqref{scalar-v-bis} and \eqref{scalar-w}, we obtain $u(R)=1,u'(R)=0$. Letting $\xi=0$ in \eqref{soluw}, one has
	\begin{align}\label{w0}
	w(0)=&e^{-cR}-(1-\Lambda)\int_{0}^Re^{-cs}v'(s) \, ds+\int_{0}^R ce^{-cs}\, ds \nonumber\\
	=&1-(1-\Lambda)\int_{0}^Re^{-cs}v'(s) \, ds,
	\end{align}
	which together with \eqref{equiv-2} gives $u(0)=\theta$. Finally, we easily retrieve $u'(0)=c\theta$.
\end{proof}

\begin{remark}\label{physical_remark} $w=u+v$ and $\widetilde w=u +\Lambda v$ are the only linear relations between $u$ and $v$ (see \cite[Remark 8.5]{BSN85}). The physical meaning of $w$ is the normalized enthalpy
(see, e.g., \cite{BL82}). In the equidiffusion case, i.e. $\Lambda=1$, $w \equiv 1$ is a Shvab-Zeldovich variable which eliminates $u=1-v$ (see Subsection \ref{Lambda=1}).
\end{remark}

\subsection{Existence of a solution}\label{existence}

Based on Proposition \ref{integral}, we need to focus only on system \eqref{scalar-v-bis}\&\eqref{equiv-1} or \eqref{scalar-v-bis}\&\eqref{equiv-2} to study system \eqref{scalar-v-bis}-\eqref{scalar-u}. More specifically, we will  prove the existence of a solution via a fixed point method.

Let $a<b$ and $I=[a,b]\subset \mathbb{R}$. We denote by $P$ the projection from $\mathbb{R}$ to $I$ defined by $P(x)= x$ if $x\in \mathrm{int}I$, $P(x)=a$ if $x\leqslant a$ and $P(x)=b$ if $ x\geqslant b$.
Obviously, $P$ is a continuous mapping.

\subsubsection{Case $0<\Lambda<1$}
We first consider the case where $0<\Lambda<1$ and prove by a fixed point argument that there exists a solution to system \eqref{scalar-v-bis}\&\eqref{equiv-2}.

For $0<\theta<1$, $0<\Lambda<1$, we set $I=[1-\theta,1-\Lambda\theta]$. For  $v_0 \in I$, let $(c(v_0),R(v_0), v(v_0;\xi))$ be the unique solution of system \eqref{scalar-v}
given by Theorem \ref{theorem_scalar}.
We define the mapping $\Phi: I \to I$ by
  \begin{equation}\label{definition_Phi}
  \Phi(v_0)=  P\left(1-\theta+(1-\Lambda)\int_0^{R(v_0)} e^{-c(v_0)s}(-v'(v_0;s))\, ds\right),
  \end{equation}
 where $P$ is the projection from $\R$ to $I$ as above.

\begin{theorem}\label{thm_fixpt_Phi}
Let $0<\theta<1$, $0<\Lambda<1$ and $I=[1-\theta,1-\Lambda\theta]$. Then, the mapping $\Phi: I \to I$ defined by \eqref{definition_Phi} has a fixed point $v^*_0$ which solves \eqref{scalar-v-bis}\&\eqref{equiv-2} with $v(0)=v^*_0$.
\end{theorem}
\begin{proof}
It follows from Corollary \ref{corollary_thm_scalar} that the mapping $\Phi: I \to I$ is continuous, hence it has a fixed point $v^*_0$. We denote by $c^*=c(v^*_0)$, $R^*=R(v^*_0)$ and $v^*(\xi)=v(v^*_0,\xi)$ the solution to \eqref{scalar-v} associated with $v^*_0$. We must prove that \eqref{equiv-2} holds, namely
\begin{equation}\label{equiv-2-bis}
v^*_0= 1-\theta+(1-\Lambda)\int_0^{R^*} e^{-c^*s}(-{v^*}'(s))\, ds.
\end{equation}

From the definition of the projection $P$ on $I$, there are 3 cases:
\begin{enumerate}[label=(\arabic*), wide, labelwidth=!, labelindent=0pt]
	\item [\rm (i)] If
$$
1-\theta+(1-\Lambda)\int_0^{R^*} e^{-c^*s}(-{v^*}'(s))\, ds\in (1-\theta,1-\Lambda\theta),
$$
then it holds that $v^*_0\in (1-\theta,1-\Lambda\theta)$ and $P$ is the identity. Therefore, \eqref{equiv-2-bis} is satisfied.

\item [\rm (ii)] If
$$
1-\theta+(1-\Lambda)\int_0^{R^*} e^{-c^*s}(-{v^*}'(s))\, ds\leqslant 1-\theta,
$$
it comes obviously
$$
(1-\Lambda)\int_0^{R^*} e^{-c^*s}(-{v^*}'(s))\, ds\leqslant 0,
$$
which is in contradiction to Theorem \ref{theorem_scalar} which states that ${v^*}'(\xi)<0$, for all $\xi\in [0,R^*)$. Then, case (ii) is ruled out.

\item [\rm (iii)] Finally, let us consider the case where
\begin{equation}\label{rightb}
1-\theta+(1-\Lambda)\int_0^{R^*} e^{-c^*s}(-{v^*}'(s))\, ds\geqslant 1-\Lambda\theta,
\end{equation}
hence $v^*_0=1-\Lambda\theta$.
It follows from Theorem \ref{theorem_scalar} that $v^*$ is decreasing and convex, hence $|{v^*}'(\xi)|\leqslant |{v^*}'(0)|$, $\forall \xi\in[0,R^*)$, therefore:
\begin{equation}\label{rightb2}
\theta\leqslant \int_0^{R^*} e^{-cs}(-{v^*}'(s))\, ds\leqslant \int_0^{R^*} \frac{c^*}{\Lambda}(1-v^*_0)e^{-c^*s}\, ds<\frac{1-v^*_0}{\Lambda}.
\end{equation}
It comes $\Lambda\theta<1-v^*_0$ which is in contradiction to $v^*_0=1-\Lambda\theta$. Then, case (iii) is also ruled out.
\end{enumerate}
\end{proof}

\subsubsection{Case $\Lambda>1$}
Likewise, we prove via a fixed point argument that there exists a solution to system \eqref{scalar-v-bis}\&\eqref{equiv-1} in the case $\Lambda>1$. Here, we take $I=[(1-\theta)/\Lambda,1-\theta/\Lambda]$. This time, we define the mapping $\Psi: I \to I$ by
\begin{equation}\label{definition_Psi}
\Psi(v_0)= P\left(\frac{1}{\Lambda}\left[1-\theta-c(v_0)(1-\Lambda)\int_0^{R(v_0)} e^{-c(v_0)s}v(v_0;s)\, ds\right]\right).
\end{equation}
\begin{theorem}\label{thm_fixpt_Psi}
	Let $0<\theta<1$, $\Lambda>1$ and $I=[(1-\theta)/\Lambda,1-\theta/\Lambda]$. Then, the mapping $\Psi: I \to I$ defined by \eqref{definition_Psi} has a fixed point $v^*_0$ which solves \eqref{scalar-v-bis}\&\eqref{equiv-2} with $v(0)=v^*_0$.
\end{theorem}
\begin{proof}
The proof follows the same pattern as the proof of Theorem \ref{thm_fixpt_Phi}.
As before, we denote by $(c^*,R^*,v^*)$ the solution to \eqref{scalar-v} associated with $v^*_0$. This time, we must prove that \eqref{equiv-1} holds, which is equivalent to
\begin{equation}\label{equiv-1-bis}
v^*_0= \frac{1}{\Lambda}\left[1-\theta-c^*(1-\Lambda)\int_0^{R^*} e^{-c^*s}v^*(s)\, ds\right].
\end{equation}
\begin{enumerate}[label=(\arabic*), wide, labelwidth=!, labelindent=0pt]
	\item [\rm (i)] If
$$
\frac{1}{\Lambda}\left[1-\theta-c^*(1-\Lambda)\int_0^{R^*} e^{-c^*s}v^*(s)\, ds\right]\in (\frac{1-\theta}{\Lambda},1-\frac{\theta}{\Lambda}),
$$
then it holds that $v^*_0\in (\mfrac{1-\theta}{\Lambda},1-\mfrac{\theta}{\Lambda})$ and $P$ is the identity. Therefore, \eqref{equiv-1-bis} is satisfied.

\item [\rm (ii)] If
$$
\frac{1}{\Lambda}\left[1-\theta-c^*(1-\Lambda)\int_0^{R^*} e^{-c^*s}v^*(s)\, ds\right]\leqslant \frac{1-\theta}{\Lambda},
$$
it comes
$$
c^*(\Lambda-1)\int_0^{R^*} e^{-c^*s}v^*(s)\, ds\leqslant 0,
$$
in contradiction to Theorem \ref{theorem_scalar} which states that $v^{*}(\xi)>0$, for all $\xi\in [0,R^*)$. Case (ii) is then ruled out.

\item [\rm (iii)] If
\begin{equation}\label{rightb3}
\frac{1}{\Lambda}\left[1-\theta-c(1-\Lambda)\int_0^{R^*} e^{-c^*s}v^*(s)\, ds\right]\geqslant 1-\frac{\theta}{\Lambda},
\end{equation}
then it holds $v^*_0=1-\frac{\theta}{\Lambda}$.
From \eqref{rightb3}, we infer that
\begin{equation}\label{rightb4}
c\int_0^{R^*} e^{-c^*s}v^*(s)\, ds\geqslant 1.
\end{equation}
On the other hand, noticing from Theorem \ref{theorem_scalar} that ${v^*}'(\xi)<0$, for all $\xi\in [0,R^*)$, it comes
\begin{equation*}
c^*\int_0^{R^*} e^{-c^*s}v^*(s)\, ds< c^*\int_0^{R^*} e^{-c^*s}v^*(0)\, ds=c^*\int_0^{R^*} e^{-c^*s}(1-\frac{\theta}{\Lambda})\, ds=(1-\frac{\theta}{\Lambda})(1-e^{-c^*R^*})<1,
\end{equation*}
in contradiction to \eqref{rightb3}.
\end{enumerate}
\end{proof}

\subsection{Final results}\label{final}
We are now in position to summarize the results of Subsections \ref{equivalence} and \ref{existence} for $0<\Lambda<1$ and $\Lambda>1$ (see Subsection \ref{Lambda=1} for $\Lambda=1$).
\begin{theorem}
Let $0<\alpha<1$, $0<\theta<1$ and $\Lambda>0$ be fixed. We define the interval $I\subset(0,1)$ as:
$I=[1-\theta,1-\Lambda\theta]$ if $0<\Lambda<1$,  $I=[(1-\theta)/\Lambda,1-\theta/\Lambda]$ if $\Lambda>1$. Then, the one-phase free boundary problem \eqref{free_boundary_problem}-\eqref{free_boundary_conditions} has a solution $(c,R,u(\xi),v(\xi))$ such that $c>0$, $0<R<+\infty$, $v(0)\in I$, $u$ and $v$ in $C^{\infty}([0,R))\cap C^{2+[\beta],\beta-[\beta]}([0,R])$, $\beta=2\alpha/(1-\alpha)$.
\end{theorem}
\begin{proof}
	According to the value of $\Lambda$ against $1$, Theorem \ref{thm_fixpt_Phi} and Theorem \ref{thm_fixpt_Psi} provide the existence of a solution to system \eqref{scalar-v-bis}\&\eqref{equiv-2}, or to system \eqref{scalar-v-bis}\&\eqref{equiv-1}, such that $v(0)\in I$. We proved in Proposition \ref{integral} that these systems are equivalent to the free boundary problem \eqref{free_boundary_problem}-\eqref{free_boundary_conditions}.
	
	Next, we apply Theorem \ref{theorem_scalar} which gives the (optimal) H\"older regularity of $v(\xi)$ at the free boundary $R$. More precisely, this result follows from Lemma \ref{holder} which is related to the asymptotic expansion of the stable manifold near the origin. Thus, the same regularity result holds for $u(\xi)$.
	\end{proof}

Finally, we are able to prove the existence of a solution to the free interface problem \eqref{TW_alpha}-\eqref{interface_conditions} thanks to Proposition \ref{proposition_simplification}. The last part of the proof of Theorem \ref{FBP_alpha} concerning the limit cases where $\alpha\to 1$ and $\alpha\to 0$ is deferred to Subsections \ref{lemma_limit_1} and \ref{lemma_limit_0}.

\section{Limit cases}
\subsection{Equidiffusion}\label{Lambda=1}
In the case $\Lambda=1$ (Lewis number equal to unity), equation \eqref{scalar-w} yields that the normalized enthalpy $w=u+v$ is equal to $1$; hence $u=1-v$. The system for $v$ reads as follows:	
\begin{equation}\label{equidiff}
\begin{cases}
v''-cv'=v^{\alpha},\quad 0<\xi<R,\\
v(0)=1-\theta,\quad
v'(0)=-c\theta,\\
v(R)=v'(R)=0.
\end{cases}
\end{equation}
It follows from Theorem \ref{theorem_scalar}, taking $v_0=1-\theta$, that system \eqref{equidiff} has a unique solution $(c,R,v(\xi))$.

\subsection{Infinite Lewis number}\label{Lambda=0}
Here, we consider the limit case $\Lambda=0$, namely an infinite Lewis number. In this case $v(0)=1$.
Integrating $cv'+v^{\alpha}=0$ yields explicitly $R=c/(1-\alpha)$ and
$$
	v(\xi)=\left(1-\xi\frac{1-\alpha}{c}\right)^{\frac{1}{1-\alpha}}.$$
It comes the formula
	\begin{equation}\label{c_formula}
	\theta =1-c\int_0^{\frac{c}{1-\alpha}}e^{-cs}\left(1-\xi\frac{1-\alpha}{c}\right)^{\frac{1}{1-\alpha}}ds,
	\end{equation}
which coincides with \eqref{equiv-1} when $\Lambda=0$.
	\begin{example}
		(i) In the case $\alpha=0$, the formula \eqref{c_formula} yields (see \eqref{speed_zero})
		\begin{equation}
		\theta=\frac{1-e^{-c^2}}{c^2}.
		\end{equation}
		(ii) In the case  $\alpha=1/2$, the formula becomes
		\begin{equation}
		\theta=\frac{1}{2}c^{-4}e^{-2c^{2}}-c^{-2}+\frac{1}{2}c^{-4}.
		\end{equation}	
	\end{example}

	However, for general $\alpha\in(0,1)$, the integral in \eqref{c_formula} can not be computed explicitly.
	For $c>0$, $0<\alpha<1$ and $0<\theta<1$, let us define:
	\begin{equation}\label{FC}
	f(c,\alpha,\theta):=\theta-1+c\int_0^{\frac{c}{1-\alpha}}e^{-cs}\left(1-s\frac{1-\alpha}{c}\right)^{\frac{1}{1-\alpha}}\, ds.
	\end{equation}
We have
	\begin{lemma} For fixed $0<\alpha<1$ and $0<\theta<1$, the equation $f(c,\alpha,\theta)=0$ has a unique positive root.
	\end{lemma}
\begin{proof}
	(i) We first look for the existence of a root $c>0$.
	It is easy to check that
	\begin{equation}\label{negative}
	\lim_{c\to 0^+}f(c,\alpha,\theta)=\theta-1.
	\end{equation}
	Next, let us prove that
	\begin{equation}\label{positive}
	\lim_{c\to +\infty}f(c,\alpha,\theta)=\theta.
	\end{equation}
	Indeed, with the change of variable $y=1-s(1-\alpha)/c$, \eqref{FC} becomes
	\begin{align}\label{fc1}
	f(c,\alpha,\theta)=\theta-1+\frac{c^2}{1-\alpha}\int_0^1 e^{-\frac{c^2(1-y)}{1-\alpha}}y^{\frac{1}{1-\alpha}}\, dy.
	\end{align}
	Next, let $z=c^2y/(1-\alpha)$. It comes:
	\begin{align*}
	f(c,\alpha,\theta)
	=\theta-1+\left(\frac{1-\alpha}{c^2}\right)^{\frac{1}{1-\alpha}}e^{-\frac{c^2}{1-\alpha}}\int_0^{\frac{c^2}{1-\alpha}} e^z z^{\frac{1}{1-\alpha}}\, dz.
	\end{align*}
	By using L'Hospital's rule, one obtains
	\begin{align*}
	\lim_{c\to+\infty}\left(\frac{1-\alpha}{c^2}\right)^{\frac{1}{1-\alpha}}e^{-\frac{c^2}{1-\alpha}}\int_0^{\frac{c^2}{1-\alpha}} e^z z^{\frac{1}{1-\alpha}}\, dz=1.
	\end{align*}
	Thus, \eqref{positive} holds. \\
	(ii) Next, to prove the uniqueness of a positive root, we compute the sign of the derivative $f_c(c,\alpha,\theta)$. To this end, we make the convenient change of variable $z=c^2(1-y)/(1-\alpha)$ in \eqref{fc1}:
	\begin{align*}
	f(c,\alpha,\theta)
	=\theta-1+\int_0^{\frac{c^2}{1-\alpha}} e^{-z} \left[1-\frac{(1-\alpha)z}{c^2}\right]^{\frac{1}{1-\alpha}}\, dz.
	\end{align*}
	Then,
	\begin{equation}
	f_c(c,\alpha,\theta)
	=\int_0^{\frac{c^2}{1-\alpha}} e^{-z} \left[1-\frac{(1-\alpha)z}{c^2}\right]^{\frac{\alpha}{1-\alpha}}\frac{2z}{c^3}\, dz>0.
	\end{equation}

\end{proof}

\subsection{Limit $\alpha \to 1$}\label{limit_1}
Here, $\Lambda>0$ and $\theta\in(0,1)$ are fixed. With each $\alpha$, $0<\alpha<1$, we associate a solution $(c,R,u(\xi),v(\xi))$ of the free boundary problem \eqref{free_boundary_problem}-\eqref{free_boundary_conditions}. However, we recall that problem \eqref{free_boundary_problem}-\eqref{free_boundary_conditions} is equivalent to the system
\begin{equation}\label{scalar-v-ter}
\begin{cases}
\Lambda v''-cv'=v^{\alpha},\quad \xi>0,\\[1mm]
v(0) \in I,\,v'(0)= -\displaystyle\frac{c}{\Lambda}(1-v(0)),\\[1mm]
v(R) = v'(R)=0,
\end{cases}
\end{equation}
together with
\begin{equation}\label{equiv-1_bis}
\theta+\Lambda v(0)=1-c(1-\Lambda)\int_0^Re^{-cs}v(s)\gap ds,
\end{equation}
 where the interval $I=[a,b]$ is given in Theorem \ref{thm_fixpt_Phi} or \ref{thm_fixpt_Psi}, according to the value of $\Lambda$, and depends only on $\Lambda$ and $\theta$.

\begin{lemma}\label{lemma_limit_1}
	Let $\Lambda>0$ and $0<\theta<1$ be fixed. Let $(\alpha_n)_{n\in\N}$ be an increasing sequence such that $0<\alpha_n<1$, $\alpha_n \to 1$ as $n \to \infty$; the corresponding solution of \eqref{scalar-v-ter}-\eqref{equiv-1_bis} is denoted by $(c_n,R_n,v_n(\xi))$. Then, $R_n \to +\infty$ as $n \to \infty$.
\end{lemma}
\begin{proof}
Assume by contradiction that the sequence $(R_n)_{n\in\mathbb{N}}$ is bounded, namely there exists $A>0$ such that $R_n \leqslant A$ for all $n\in\N$. In view of $0\leqslant v_n(\xi)< 1$ for all $0\leqslant \xi \leqslant R_n$ and formula \eqref{integral_v}, it holds
$$
c_n= \int_0^{R_n} v^{\alpha_n}_{n}(\xi) \gap d\xi,
$$
hence $c_n \leqslant R_n \leqslant A$. In addition, according to \eqref{estimate_c},
\begin{equation*}
c^2_n \geqslant \frac{2\Lambda}{1+\alpha_n}v_n^{1+\alpha_n}(0),
\end{equation*}
$v_n(0) \geqslant a=\inf I>0$. Hence, the sequence $(c_n)_{n\in\N}$ is bounded from below by $a\sqrt{\Lambda}$.

By abuse of notation,  we also denote by $v_n(\xi)$ its extension by $0$ to the interval $[0,A]$.
The system for $(c_n, v_n(\xi))$ reads
\begin{eqnarray}\label{scalar_FBP_n}
\left\{
\begin{array}{ll}
&\Lambda  v''_{n}-c_n v'_{n}=v_{n}^{\alpha_{n}} , \quad 0<\xi<A, \\[2mm]
& v_{n}(0) \in I, \quad v'_{n}(0)= -\displaystyle\frac{c_n}{\Lambda}(1-v_n(0)),\\[2mm]
&v_n(A)=v'_{n}(A)=0.
\end{array}
\right.
\end{eqnarray}

 Then, it is easy to see that the sequence $(v_n)_{n\in\N}$ is bounded in the space $H^2([0,A])$. We extract a convergent subsequence  $(c_{n'},v_{n'})_{n'\in\N}$ such that $c_{n'} \to c_{\infty} >0$, $v_{n'} \to v_{\infty}$ in the space $C^1([0,A])$.
We pass to the limit as $n' \to \infty$, $\alpha_{n'} \to 1$: there exists a non-negative function $\chi$ such that $(v_{n'})^{\alpha_{n'}} \rightharpoonup \chi$ weakly in $L^2(0,A)$. At the limit it holds in the distribution sense in the interval $(0,A)$:
\begin{equation*}
\Lambda v''_{\infty}
-c_{\infty} v'_{\infty}=\chi.
\end{equation*}
Moreover, at the limit,
\begin{equation*}
v_{\infty}(0)\in I, \quad v'_{\infty}(0)= -(c_{\infty}/\Lambda)(1-v_{\infty}(0)), \quad v_{\infty}(A)=v'_{\infty}(A)=0.
\end{equation*}
We observe that $v_{\infty}$ belongs to $C^1([0,A])$ and is convex, non-increasing:
there exists $0<R_{\infty} \leqslant A$ such that $v_{\infty}(\xi)>0$ for $0<\xi<R_{\infty}$, $v_{\infty}(\xi)=0$ for $R_{\infty} \leqslant \xi \leqslant A$.
At fixed $\xi \in (0, R_{\infty})$, $(v_{n'}(\xi))^{\alpha_{n'}} \to v_{\infty}(\xi)$ as $n' \to \infty$.
Therefore, it comes $\chi(\xi)=v_{\infty}(\xi)$ whenever $0<\xi<{R}_{\infty}$.
The system eventually reads:
\begin{eqnarray}\label{scalar_FBP_infty}
\left\{
\begin{array}{ll}
&\Lambda  v''_{\infty}
-c_{\infty} v'_{\infty}= v_{\infty}, \quad 0<\xi<R_{\infty},\\[2mm]
&v_{\infty}(0)\in I, \quad v'_{\infty}(0)= -\displaystyle\frac{c_{\infty}}{\Lambda}(1-v_{\infty}(0)),\\[2mm]
&v_{\infty}({R}_{\infty})= v'_{\infty}({R}_{\infty})=0,
\end{array}
\right.
\end{eqnarray}
which has no solution for finite ${R}_{\infty}$ (see also \cite[Lemma 3.6 case (i)]{BSN85}).

As a consequence, the sequence $(R_n)_{n\in\N}$ is unbounded when $\alpha_n \to 1$. We may extract a subsequence $n'' \to \infty$ such that $R_{n''} \to +\infty$. Finally, repeating the process, $(R_n)_{n\in\N}$ goes to $+\infty$.
\end{proof}

\subsection{Limit $\alpha \to 0$}\label{limit_0}
As in the previous subsection, $\Lambda>0$ and $\theta\in(0,1)$ are fixed. With each $\alpha$, $0<\alpha<1$, we associate a solution $(c,R,v(\xi))$ of the problem \eqref{scalar-v-ter}-\eqref{equiv-1_bis}.
\begin{lemma}\label{lemma_limit_0}
	Let $\Lambda>0$ and $0<\theta<1$ be fixed. Let $(\alpha_n)_{n\in\N}$ be a decreasing sequence such that $0<\alpha_n<1$ and $\alpha_n \to 0$ as $n \to \infty$; the corresponding solution of \eqref{scalar-v-ter}-\eqref{equiv-1_bis} is denoted by $(c_n,R_n,v_n(\xi))$. Then, $(c_n,R_n,v_n(\xi))_{n\in\N}$ converges to the unique solution $(c,R,v(\xi))$ of the limit system
\begin{eqnarray}\label{scalar_v_0}
\left\{
\begin{array}{ll}
&\Lambda v''-cv'=1 , \quad 0<\xi<R, \\[2mm]
& v(0) \in I,\, v'(0)= -\displaystyle\frac{c}{\Lambda}(1-v(0)),\\[2mm]
&v(R)=v'(R)=0,
\end{array}
\right.
\end{eqnarray}
together with
\begin{equation}\label{equiv-1_bis_0}
\theta+\Lambda v(0)=1-c(1-\Lambda)\int_0^{R}e^{-cs}v(s)\gap ds,
\end{equation}
where $c=R$ is uniquely given by \eqref{speed_zero}, see Subsection \ref{Lambda=0} (i).	
\end{lemma}
\begin{proof}
The proof follows most of the lines of the proof of Lemma \ref{limit_1}.
This time, we must show that the sequence  $(R_n)_{n\in\mathbb{N}}$ is bounded.
We take advantage of formula \eqref{eq-bounds-v0} which reads:
\begin{equation}\label{uniform_R}
\frac{(2(1+\alpha_n)\Lambda)^{1/2}}{1-\alpha_n}a^{\frac{1-\alpha_n}{2}}<R_n< \frac{2\Lambda^{1/2}A(\alpha_n)}{1-\alpha_n}\frac{b^\frac{1-\alpha_n}{2}}{1-b}.
\end{equation}
Here, $A(\alpha_n)={2^{3/2}}/{(1+\alpha_n)^{1/2}}<2^{3/2}$. For simplicity we assume $0<\alpha_n \leqslant 1/2$ for all $n\in\N$. It comes the new definition of $A$:
\begin{equation}\label{uniform_R_bis}
{(2a\Lambda)^{1/2}}<R_n< A=2^{7/2}\Lambda^{1/2}\frac{b^{1/4}}{1-a}.
\end{equation}
Hence, the sequence $(c_n)_{n\in\N}$ is also bounded from above and from below, as in Lemma \ref{lemma_limit_1}.

Next, we mimic the proof of Lemma \ref{lemma_limit_1}. By abuse of notation,  we also denote by $v_n(\xi)$ its extension by $0$ to the interval $[0,A]$.
The system for $(c_n, v_n(\xi))$ reads
\begin{eqnarray}\label{scalar_FBP_n}
\left\{
\begin{array}{ll}
&\Lambda  v''_{n}-c_n v'_{n}=v_{n}^{\alpha_{n}} , \quad 0<\xi<A, \\[2mm]
& v_{n}(0) \in I, \quad v'_{n}(0)= -\displaystyle\frac{c_n}{\Lambda}(1-v_n(0)),\\[2mm]
&v_n(A)=v'_{n}(A)=0.
\end{array}
\right.
\end{eqnarray}
The sequence $(v_n)_{n\in\N}$ is bounded in the space $H^2([0,A])$. We extract a convergent subsequence  $(c_{n'},v_{n'})_{n'\in\N}$ such that $c_{n'} \to c >0$, $v_{n'} \to v$ in the space $C^1([0,A])$.
We pass to the limit as $n' \to \infty$, $\alpha_{n'} \to 0$: there exists a non-negative function $\chi$ such that $(v_{n'})^{\alpha_{n'}} \rightharpoonup \chi$ weakly in $L^2(0,A)$. At the limit it holds in the distribution sense in the interval $(0,A)$:
\begin{equation*}
\Lambda v''
-c v'=\chi.
\end{equation*}
As in the proof in Lemma \ref{lemma_limit_1}, there exists $0<R\leqslant A$ such that $v(\xi)>0$ for $0<\xi<R$, $v(\xi)=0$ for $R \leqslant \xi \leqslant A$.
The only difference is the following: at fixed $\xi \in (0,R)$, $(v_{n'}(\xi))^{\alpha_{n'}} \to 1$ as $n' \to \infty$.
Therefore, $\chi(\xi)=1$ whenever $0<\xi<R$.
The system eventually reads:
\begin{eqnarray}\label{scalar_FBP_infty_bis}
\left\{
\begin{array}{ll}
&\Lambda  v''
-c v'= 1, \quad 0<\xi<R,\\[2mm]
&v(0)\in I, \quad v'(0)= -\displaystyle\frac{c}{\Lambda}(1-v(0)),\\[2mm]
&v(R)= v'(R)=0.
\end{array}
\right.
\end{eqnarray}

The next step is to take the limit in the formula
\begin{equation*}\label{equiv-1_bis_n}
\theta+\Lambda v_{n'}(0)=1-c_{n'}(1-\Lambda)\int_0^{A}e^{-c_{n'}s}v_{n'}(s)\gap ds.
\end{equation*}
By Lebesgue's Theorem, it comes:
\begin{equation}\label{equiv-1_bis_infty}
\theta+\Lambda v(0)=1-c(1-\Lambda)\int_0^{R}e^{-cs}v(s)\gap ds.
\end{equation}

It follows from the uniform convergence of $v_{n'}(\xi)$ to $v(\xi)$ on the interval $[0,A]$ that $R_{n'} \to R$. Finally, because the system \eqref{scalar_FBP_infty_bis}-\eqref{equiv-1_bis_infty} has a unique solution, the entire sequence converges to $(c,R,v(\xi))$.

The proof is thus complete.
\end{proof}


\section{Acknowledgment} The authors want to thank Peter Gordon for bringing Chung K. Law's book to our attention and for fruitful discussions.
P.S. and L.Z. would like to thank the National Natural Science Foundation of China (No. 11771336).

\appendix
\section{About the Poincar\'e-Bendixson Theorem}\label{subsect-summary}
There are many textbooks on the qualitative theory of vector fields (see, e.g., Hartman \cite{H}, Lefschetz \cite{L}, Arnold \cite{A}). First versions of the  Poincar\'e-Bendixson Theorem were given in the historical works of Poincar\'e (see \cite{P}) and Bendixson (see \cite{B}). Here, we state a simplified modern version of the Poincar\'e-Bendixson Theorem (see Theorem \ref{P-B-theorem} below). A more complete version may be found in \cite{L}.

To begin with, for the convenience of the reader, we recall some basic definitions.

Let us consider a topological flow $\varphi(t,m)$ on a manifold $M$. For each $m\in M$ the curve $t\mapsto \varphi(t,m)$, parametrized by the time $t$, is the trajectory by $m$ defined on a maximum open interval $(-\tau_-(m),+\tau_+(m))$ neighborhood of $0$ and such that $\varphi(0,m)=m$ (one may have $\tau_+(m)=+\infty$ and/or  $-\tau_-(m)=-\infty)$. The {\it orbit} $\gamma_m$ is the oriented geometrical image of the trajectory by $m$ (the orbit is positively oriented by the time direction but the time parametrization is omitted). We can also consider the positive half-orbit
$\gamma^+_m= \varphi([0,\tau_+(m)),m)$ and the negative half-orbit  $\gamma^-_m= \varphi([0,-\tau_-(m)),m)$.

The following limit sets have been intensively studied and used in the qualitative theory of dynamical systems (see, e.g., \cite{L},\cite{A},\cite{BS},\cite{DF}):
\begin{definition}\label{limit-sets} Let $\gamma=\gamma_m$ be an orbit of a continuous flow $\varphi(t,m)$  on a manifold $M$. The $\omega$-limit set of $\gamma$ is the set
	$\omega(\gamma)$ of points $p\in M$ such that, for $m\in \gamma,$  there exists a sequence of times $(t_i)\rightarrow +\tau_+(m)$ with the property that $\varphi(t_i,m)\rightarrow p,$  for $i\rightarrow +\infty.$ In a similar way, one defines the $\alpha$-limit set $\alpha(\gamma)$ by considering the sequences $(t_i)\rightarrow -\tau_-(m).$ One can also write $\omega(\gamma^+_m)$ and $\alpha(\gamma^-_m)$ for $\omega(\gamma)$ and $\alpha(\gamma)$ respectively. These limit sets does not depend on the choice of a particular point $m\in \gamma$.
\end{definition}

We recall that a vector field $X$ on a $n$-manifold $M$ is defined, on each open chart $\Omega$ with coordinates $x=(x_1,\ldots,x_n)$, as a first order differential operator:
\begin{equation*}
X(x)=A_1(x)\frac{\partial}{\partial x_1}+\cdots +A_n(x)\frac{\partial}{\partial x_n}.
\end{equation*}
The vector field is said of class $C^k$ if the components  $A_i(x)$ are of this class on each chart $\Omega$.  A vector field may depend on a  parameter $\lambda$,  when the components in each chart depend on this parameter. A family of vector fields $X_\lambda$  is of class $C^k$, if the components $A_i(x,\lambda)$ are of this class as functions of $(x,\lambda)$.

We also recall a basic result which is a corollary of the (real) Cauchy Theorem:
\begin{theorem} \label{th-Cauchy}
	Let $X_\lambda$ be a family of vector fields defined on an open set $\Omega$ in $\R^n,$ with
	$\lambda\in  W,$ an open set of $\R^p.$ This family is assumed of class of differentiability
	$\mathcal{C}^k, k\geqslant 1,$ or of analytic class  (we mean that it is represented by a map
	$\Omega\times W\rightarrow \R^n$ of this given class).
	We call $\varphi(t,m,\lambda)$ its flow and consider $(t_0,m_0,\lambda_0)\in \R\times \Omega\times W$ such that the segment of $\varphi([0,t_0],m_0,\lambda_0)$ is contained in $\Omega.$
	Then the flow $\varphi(t,m,\lambda)$ is defined in a neighborhood $U$ of
	$(t_0,m_0,\lambda_0)$ in
	$\R\times \Omega\times W$ and has on $U$ the same class of differentiability as the family $X_\lambda.$
\end{theorem}

\begin{theorem}{(\bf Poincar\'e-Bendixson Theorem)}\label{P-B-theorem}
	One considers a $C^1$ vector field $X$ defined on an open set $\Omega$ of $\R^2$.  One assumes that a positive half-orbit $\gamma^+_m$ of the flow has a compact closure in $\Omega$. Then, $\omega(\gamma^+_m)$
	contains a singular point or is a closed (periodic) orbit. As the $\alpha$-limit sets of $X$ are the $\omega$-limit sets of $-X,$ one has a similar result for the negative half-orbits $\gamma^-_m$ of $X$.
\end{theorem}

We have the following corollary of the Poincar\'e-Bendixson Theorem which will be the version used in this paper:
\begin{corollary}\label{P-B-corollary}
	One considers a $\mathcal{C}^1$ vector field $X$ defined on an open set $\Omega$ of $\R^2,$ with flow $\varphi(t,m).$ One assumes that $\Omega$ has a compact closure $\bar \Omega$ in $\R^2$ and we write $\partial \Omega=\bar\Omega\setminus\Omega$ for its boundary.  One also assumes  that
	$\Omega$ is simply connected and that $X$ has no singular point (in $\Omega$). Let $m$ be any point in $\Omega.$ Then there exists a sequence $t_i\rightarrow \tau_+(m)$ and a point $p\in \partial\Omega$ such that $\varphi(t_i,m)\rightarrow p$ when $i\rightarrow +\infty.$ A similar result is true for the negative times.
\end{corollary}
We refer to \cite{RR} for a detailed proof of Corollary \ref{P-B-corollary}.
\section{Proof of Proposition \ref{prop-phase-portrait}}
The vector field $X^E_c$ has a unique singular point at the origin $O$, two trajectories whose $\omega$-limit   and two trajectories whose $\alpha$-limit set is the origin  $O$. Each other trajectory goes from infinity to infinity.  One can expect that these properties imply that  $X^E_c$ is topologically equivalent to an hyperbolic linear saddle vector field.  Unfortunately,  this is not true in general as that was already detected by Bendixson in \cite{B}.
Then, in the proof given below, we will have to use also more specific properties of the vector field $X^E_c$.
\subsection{Construction of the equivalence on $Q$}\label{subsect-equivalence-Q}
In order to simplify the notation, and as $c,\alpha,\Lambda$ are assumed to be fixed,  we will write $X$ for the vector $X_c$.  We will denote by $L$ the unique stable manifold of $X$
 and $y(x)$ the  function whose graph is $L$; the latter function reads $y_c(x)$ in Section \ref{sect-topological-approach}. We also denote by  $L^0$ the stable manifold of the linear field
 $X^0=y\frac{\partial }{\partial x}+x\frac{\partial }{\partial y}$, restricted to  the quadrant $Q$,  which is the graph  $\{y=y^0(x)=-x\}$.

 The stable manifold $L:=\{y=y(x)\}$ splits $Q$ into two invariant regions:
  $Q\setminus L=Q^+\cup Q^-$ where $Q^+=\{x>0, y(x)<y\leqslant 0\}$ and
  $Q^-=\{x\geqslant 0, y< y(x)\}$.
   To begin with, we want to establish some properties for the orbits of $X$, which are trivially verified by the orbits of the linear vector field $X^0$.

 \begin{lemma}\label{lem-orbits}
 We consider the vector field $X$ defined  on $Q$.\\
\noindent (1) Any orbit $\gamma$ contained into $Q^+$ is a graph
  $\{y=y_\gamma(x)\}$ over an interval $[x_\gamma,+\infty)$  for some $x_\gamma$ such that  $0<x_\gamma<x$.\\
\noindent (2) Any orbit contained into $Q^-$ is a graph  $\{y=y_\gamma(x)\}$ over $[0,+\infty)$. Moreover if $(x_0,y_0)\in \gamma$, we have that: $y_0-c/\Lambda x_0\leqslant y_\gamma(0)<0$.
 \end{lemma}
 \begin{proof}
 The point (1) follows from the results  of  Section \ref{sect-topological-approach}:  a  trajectory starting at a point $(x,y)\in Q^+$ cannot tend toward $O$ because $L$ is the unique stable manifold in $Q$  by Lemma \ref{lem-exist-global-s-m}; then, as  a consequence of  Lemma \ref{lem-exist-local-s-m}, this trajectory must attain a point
 $(x_\gamma,0)$ with $0<x_\gamma<x$. Finally, as in the proof used in Lemma \ref{lem-exist-global-s-m} for the stable manifold, it is easy to prove  that the corresponding orbit $\gamma$ is a graph $\{y=y_\gamma(x)\}$ above $[x_\gamma,+\infty)$.

 We now consider the orbits in $Q^-$. We  first claim that, in the interior of  the whole quadrant $Q$, the vector field $X$ is transverse to the direction of the vector $Z=(1,c/\Lambda)$ and directed toward the left. To see this, we  consider the orthogonal vector field
  $Z^\perp=(-c/\Lambda,1)$ and compute the scalar product
  \begin{equation*}
  <Z^\perp,X>=-\frac{c}{\Lambda}y+\Big(\frac{c}{\Lambda}y+\frac{1}{\Lambda}x^\alpha\Big)=\frac{1}{\Lambda}x^\alpha,
  \end{equation*}
  which is strictly positive if $x>0.$ The claim follows.

 Consider  now an orbit $\gamma$ in $Q^-.$  By an argument similar to the one used in Lemma \ref{lem-exist-global-s-m} it is easy to see that $\gamma$ is a graph above some maximal interval $I$  in the $Ox$-axis. We want firstly to prove that if $(x_0,y_0)\in \gamma$, then $[0,x_0]\subset I$ (this means that $I$ is closed on the left with $0$ as end point).  To show this point we consider the  {\it curved rectangle}  $\mathcal{R}=[O,A,B,C]$, where $A=(x_0,y(x_0)), B=(x_0,y_0)$ and  $C$  belongs to  the $Oy$-axis (observe that necessarily we have that $y_0<y(x_0)$). The side $[O,A]$ is an arc in $L,$  the three other sides are linear segments  and $[B,C]$ as a direction parallel to the vector $Z$ above, see Figure \ref{fig-rect-R}. The vector field X is tangent to $[O,A],$ transverse toward the left along
 $[A,B], [B,C)$  and $[C,O)$: this means that $X$ is entering $\mathcal{R}$ along $[A,B]\cup [B,C)$ and is going out $\mathcal{R}$ along $[C,O)$. Now, as in the point (1), we can use the results of Section  \ref{sect-topological-approach} to obtain that the trajectory of the point
 $(x_0,y_0)$ must reach a point of $[C,O)$ in a finite time.

 This implies that $\gamma$ is a graph $\{y=y_\gamma(x)\}$ over a maximal interval
 $[0,x_+)$. Moreover, as the ordinate of $C$ is equal to $y_0-c/\Lambda x_0$ that is less than $y_\gamma(0)$, we have that $y_\gamma(0)\geqslant y_0-\frac{c}{\Lambda}x_0,$
 as stated in the Lemma.

  We prove  now  that $x_+=+\infty$   by using the same idea as in the proof
 of Lemma \ref{lem-unicity}.  We have that $d\Big(y(x)-y_\gamma(x)\Big)/dx<0$ on
 $(0,x_+),$ the common domain of existence of $y(x)$ and $y_\gamma(x)$.  This inequality implies that, for all $x<x_+$, we have that the function $y(x)-y_\gamma(x)$ is decreasing and then that:
 \begin{equation*}
 y_\gamma(x)>y(x)-y(x_0)+y_\gamma(x_0).
 \end{equation*}
  Passing to the limit $x_+$ this implies that the $\alpha$-limit set of $\gamma$ is contained in the closed interval $J=\{x_+\}\times [y(x_+)-y(x_0)+y_\gamma(x_0),0]$ and then is non-empty and must contain a singular point, as consequence of the Poincar\'e-Bendixson Theorem.
 As $X$ has no singular point in a neighborhood of $J$, this is impossible.
 \end{proof}
 \vskip5pt
\begin{figure}[htp]
\begin{center}
   \includegraphics[scale=0.6]{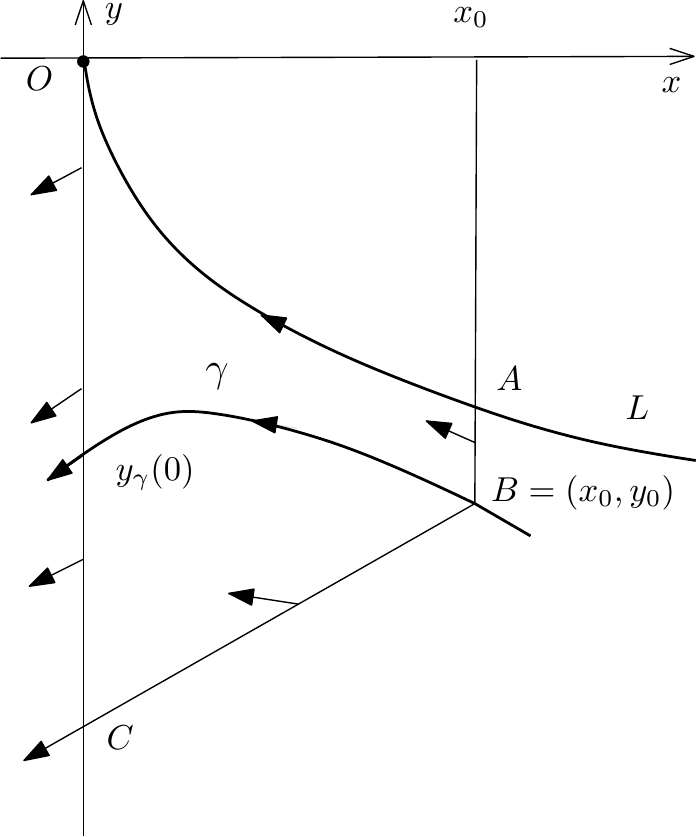}\\
  \caption{\it Rectangle $\mathcal{R}$}\label{fig-rect-R}
  \end{center}
\end{figure}
\vskip5pt
 Lemma \ref{lem-orbits} allows to define  two transition functions: for
$(x,y)\in Q^+$, the function $T_x(y)$ which is the value $x_\gamma$ defined in Lemma \ref{lem-orbits} for the orbit $\gamma$
  through the point $(x,y)$; for
  $(x,y)\in Q^-$, the function $S_x(y)$ where $(0,S_x(y))$ is the point reached  on the axis $Oy$ by the orbit $\gamma$ through $(x,y)$ (see Figure \ref{fig-H}). Clearly, these functions are analytic on $Q^+$ and $Q^-$ respectively. The following lemma shows that they can be continuously extended on $L$ with the common value $0$:

  \begin{lemma}\label{lem-T-S}
  Let $\bar Q^+=Q^+\cup L= \{x\geqslant 0,y(x)\leqslant y\leqslant 0\}$ and
  $\bar Q^-=Q^-\cup L= \{x\geqslant 0,y\leqslant y(x)\}$ be the closure of $Q^+$ and $Q^-$. Then $T_x(y)$  and $S_x(y)$ can be  continuously extended  by $0$ on $L$.
 \end{lemma}
 \begin{proof} To prove the  continuity along $L$,  we will proceed in three steps.

 {\it Step 1. We fix a value $x_0>0$ and prove that $T_{x_0}(y)$ and
 $S_{x_0}(y)$ tends toward $0$ when $y$ tends toward $y(x_0)$.} We begin with $T_{x_0}(y).$ Let us observe that $y\mapsto T_{x_0}(y)$ is decreasing from $T_{x_0}(0)=x_0$  when $y$ decreases from $y=0$. Let $x_T\geqslant 0$ be the limit when $y\rightarrow y(x_0)$. We claim that $x_T=0$. If not, the orbit of $(x_T,0)$ for the vector field $-X$ reaches the vertical axis $\{x=x_0\}$ at some point $(x_0,y_T)$ with $y(x_0)<y_T<0$. This is clearly in contradiction with the definition of $x_T$.  As a  conclusion, the function $y\mapsto T_{x_0}(y)$ extends continuously by $0$ at $y=y(x_0)$. We have exactly the same proof for $S_{x_0}(y)$.

{\it Step 2.  The extensions are continuous on $L\setminus \{O\}$ as functions of the two variables $(x,y)$.} Let us prove the continuity of the extension of $T_x(y)$ at a point $(x_0,y(x_0))\in L$ with  $x_0>0$. The idea is that we can compute
 $T_x$ in terms  of $T_{x_0}$  and the transitions along the flow of $X$. If we denote by
 $\Sigma_x$ the vertical  section by $x$ (parameterized by $y$), there is an analytic transition along the flow of $X$:  $y\rightarrow T^{x_0}_x(y)$ from $\Sigma_x$
  to  $\Sigma_{x_0}$ such that
 $T^{x_0}_x(y(x))=T_x(y(x_0))$. Moreover it is clear that
 \begin{equation*}
 T_x(y)=T_{x_0}\circ T_x^{x_0}(y).
 \end{equation*}
 As the extension of $T_{x_0}$ is continuous in $y$ at $y(x_0)$ and $T_x^{x_0}(y)$ is analytic  in $(x,y)$ at $(x_0,y(x_0))$, the above formula shows that the extension of $T_x(y)$  is continuous in $(x,y)$ at $(x_0,y(x_0))$.  Exactly the same proof can be made for $S_x(y)$.

 {\it Step 3. Continuity of the extension at the origin $O$.} The fact that $T_x(y)\rightarrow 0$ for $(x,y)\rightarrow (0,0)$ in $\bar Q^+$ follows from the inequality $0\leqslant T_x(y)\leqslant x$. For the function $S_x(y)$, we apply the inequality
  $|S_x(y)|\leqslant |y|+\frac{c}{\Lambda}|x|$ which is an equivalent form of the  inequality states in the point (2) of Lemma \ref{lem-orbits}.
 \end{proof}
\textit{From now on and for simplicity, we will also write $T_x(y)$ and $S_x(y)$ for the continuous extensions of these functions on $\bar Q^+$ and $\bar Q^-$, i.e. along the curve $L$.}
 For each $x>0$ and as it is monotonic,  the map $y\mapsto T_x(y)$ is an homeomorphism from the interval $[0,y(x)]$ to the interval $[x,0]$.  Also, for each $x\geqslant 0$,  the map $y\mapsto S_x(y)$ is an homeomorphism from the interval $(-\infty,y(x)]$ to the interval $(-\infty,0]$. Moreover we have proved in Lemma \ref{lem-T-S} that these $x$-families of homeomorphisms are continuous with the limits $S_0(y)\equiv y$ and $T_0(0)=0$ (the interval $[0,y(0)]$ is reduced to the point $\{0\}$. This allows to define the inverse homeomorphisms $T^{-1}_x(z)$ and $S^{-1}_x(z)$ defined respectively for $0\leqslant z\leqslant x$ and $z\in (-\infty,0].$  These $x$-families are clearly analytic for $z\not =0.$ we prove now that they are  continuous in $(x,z)$ at $z=0$:
 \begin{lemma}\label{lem-T(-1)-S(-1)}
 The inverse maps  $T^{-1}_x(z)$ and $S^{-1}_x(z)$ have a  continuous extension for $z=0$ by $T^{-1}_0(0)=0$ and $S^{-1}_0(z)\equiv z$.
 \end{lemma}
 \begin{proof}
 The proof is indeed much easier than the one given in Lemma \ref{lem-T-S} for the direct maps. Now, we can take advantage of the contraction property of the vertical direction, along the flow followed in the $x$-direction, already proved in Lemma \ref{lem-unicity} and used above in Lemma \ref{lem-orbits}.  This property implies that $|T^{-1}_x(z)-y(x)|\leqslant |y(z)|$ and that $|S^{-1}_x(z)-y(x)|\leqslant |z|$. The result follows directly for $z\rightarrow 0$.
 \end{proof}

 As $X^0$ is just a particular case, we have transition functions $T^0_x(y)$ and $S^0_x(y)$ for it, which verify the statements of Lemmas
\ref{lem-orbits}, \ref{lem-T-S} and \ref{lem-T(-1)-S(-1)} (in fact, as $T^0_x(y)= \sqrt{x^2-y^2}$  and  $S^0_x(y)=-\sqrt{y^2-x^2}$,  we can  verify directly their properties).
Using the transitions functions of $X^0$ and $X$ we can now  construct on $Q$ a topological equivalence between these vector fields:
\begin{proposition}\label{prop-Q-equivalence}
We define a function $h_x(y)$ on $Q$ by  $h_x(y)=T^{-1}_x\circ T^0_x(y)$ if $(x,y)\in \bar Q^+$ and by  $h_x(y)=S^{-1}_x\circ S^0_x(y)$ if $(x,y)\in \bar Q^-$ (see Figure \ref{fig-H}). Then, the map $H(x,y)=(x,h_x(y))$ is an homeomorphism of $Q$, which is the identity on the boundary $\partial Q$ ($H(x,0)=(x,0)$ for $x\geqslant 0$ and $H(0,y)=(0,y)$ for $y\leqslant 0$).
Moreover, $H$ is a topological equivalence between $X^0$ and $X,$ i.e. $H$ sends the orbits of $X^0$ onto the orbits of $X$.
\end{proposition}
\begin{proof} As $T^0_x$ and $T^{-1}_x$ are homeomorphisms from $[0,y_0(x)]$ to $[x,0]$ and from $[x,0]$ to $[0,y(x)]$ respectively, the map $h_x(y)$ is an homeomorphism from $[0,y_0(x)]$ to $[0,y(x)]$.  For similar reasons we have that $h_x(y)$ is an homeomorphism from
$(-\infty,y_0(x)]$ to $(-\infty,y(x)]$. As these two definitions of $h_x(y)$ coincide at $y_0(x)$, they define a global homeomorphism also denoted by $h_x(y)$ from $(-\infty,0]$ to itself.
The continuity  on $Q$ of the map $(x,y)\mapsto h_x(y)$ follows from the  continuity of the transition functions and their inverse in terms of their two variables, proved in Lemmas \ref{lem-T-S} and \ref{lem-T(-1)-S(-1)}. As $h^{-1}_x(y)=(T^0)^{-1}_x\circ T_x(y)$ on $\bar Q^+$ and
$h^{-1}_x(y)=(S^0)^{-1}_x\circ S_x(y)$ on $\bar Q^-$, the same argument shows that
$(x,y)\mapsto h^{-1}_x(y)$ is a continuous family of homeomorphisms from
$(-\infty,0]$ to itself.  Clearly, the map $(x,y)\mapsto (x, h^{-1}_x(y))$ is the inverse of the map $H(x,y)=(x,h_x(y))$. As these two maps are continuous,  we have that $H$ is an homeomorphism of $Q$ to itself. Moreover it is clear that $H$ is equal to the identity on $\partial Q$.

By definition of the transition maps, the homeomorphism $H$ sends each orbit of $X^0$ onto the orbits of
of $X$   (see Figure \ref{fig-H}; in particular $H$ sends $L^0$ onto $L$). We have proved that $H$ is a topological  equivalence on $Q$.
\end{proof}
\vskip5pt
\begin{figure}[htp]
\begin{center}
   \includegraphics[scale=0.7]{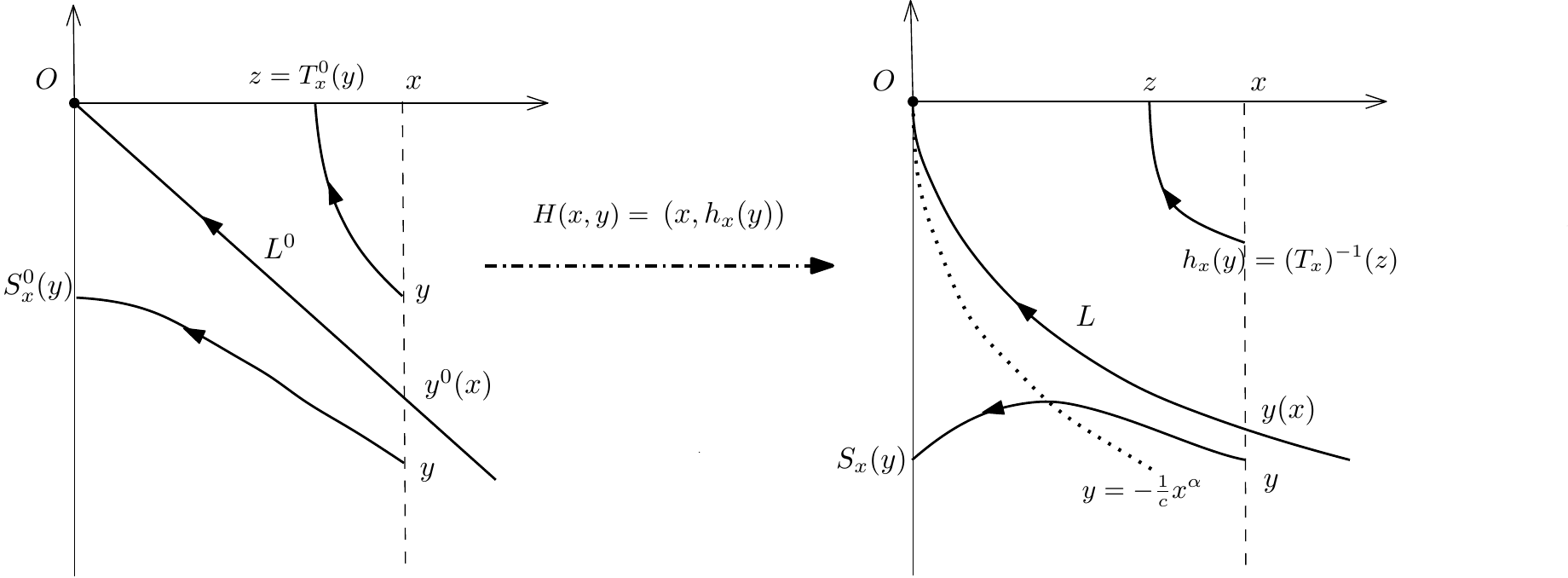}\\
 \caption{\it The topological equivalence $H$}\label{fig-H}
  \end{center}
\end{figure}
\vskip5pt
\subsection{Extension of the equivalence to $\R^2$}
 We can obtain   very similar properties for $X_c$ on the quadrant $Q'=\{x\geqslant 0, y\geqslant 0\}$ as the ones proved   for $X_c$ on $Q,$  in Section \ref{sect-topological-approach} and in the above subsection \ref{subsect-equivalence-Q}.
A noticeable difference is that the slope of $X_c$ is greater than $c/\Lambda$  everywhere in $Q'$. As a consequence, we can find a unique unstable manifold at the origin which is
a graph $x=x_c(y)$ and more generally each orbit $\gamma$ in $Q'$ is a graph $x_\gamma(y)$ above an interval $[y_\gamma,+\infty)$ with $y_\gamma\geqslant 0$.

 We can repeat the proof given in Subsection \ref{subsect-equivalence-Q} in order to construct an homeomorphism $H'$ of $Q$  of the form $(h'_y(x),y)$ where $h'_y(x)$ and
 $(h'_y)^{-1}(x)$ are continuous families of homeomorphisms and such that $H'$ is equal to the identity on $\partial Q'$. This homeomorphism can be glued up with $H$ to give a topological equivalence $\hat H$ of $X^0$ and $X_c$ on the half-plane $\{x\geqslant 0\}$, which is equal to the identity along the $Oy$-axis.

 To extend $\hat H$ to the whole plane, we use the symmetry property of $X^E_c$, already mentioned in Subsection \ref{subsect-phase-portrait}: $X_c^E(-x,-y)=-X_c(x,y)$.
 For each $x\leqslant 0,$ we define $\hat H(x,y)$ by $\hat H(x,y)=-\hat H(-x,-y).$ This extends $\hat H$ in the half-space $\{x\leqslant 0\}$ and the two definitions coincide along $\{x=0\}$  with the identity. Then, this defines a global topological equivalence $\hat H$ on  $\R^2$ between  $X^0$ and $X^E_c$ (this construction takes into account the fact  that the orbits of $X^0$ and $X^E_c$ are crossing the $Oy$-axis at the points $(0,y)$ such that $y\not =0$).


\appendix

\end{document}